\documentclass[12pt,a4paper,reqno]{amsart}

\usepackage[margin = 1in]{geometry}
\usepackage{a4wide}
\usepackage{amssymb,amsmath,amsthm,mathrsfs}
\usepackage{graphicx}
\usepackage{caption}
\usepackage{float}
\usepackage{colortbl}
\usepackage{enumerate}
\usepackage{upref}
\usepackage{tikz}
\usepackage{pgfplots}
\pgfplotsset{compat=1.11}
\usepgfplotslibrary{fillbetween}
\usetikzlibrary{intersections}
\usetikzlibrary{arrows,backgrounds,patterns}
\pgfdeclarelayer{bkgr}
\pgfdeclarelayer{bkgrd}
\pgfsetlayers{bkgrd,bkgr,main}
\usetikzlibrary{arrows}
\usetikzlibrary{decorations.pathmorphing}
\usetikzlibrary{decorations.pathreplacing}
\usepackage[T2A,T1]{fontenc}
\DeclareSymbolFont{cyrillic}{T2A}{cmr}{m}{n}
\DeclareMathSymbol{\D}{\mathalpha}{cyrillic}{196}
\DeclareMathOperator*{\esssup}{ess\,sup}
\DeclareMathOperator*{\essinf}{ess\,inf}
\usepackage{bbm}
\usepackage[bookmarks=false]{hyperref}
\theoremstyle{plain}
\newtheorem{theorem}{Theorem}[section]
\newtheorem{lemma}[theorem]{Lemma}

\newtheorem{proposition}[theorem]{Proposition}

\theoremstyle{definition}

\theoremstyle{remark}
\newtheorem{remark}[theorem]{Remark}

\usepackage{enumitem}
\makeatletter
\def\namedlabel#1#2{\begingroup
   #2%
 \def\@currentlabel{#2}%
   \phantomsection\label{#1}\endgroup
}

\newcommand{\floor}[1]{\lfloor#1\rfloor}
\newcommand{\set}[1]{\left\{#1\right\}}
\newcommand{\abs}[1]{\left|#1\right|}
\usepackage[T2A,T1]{fontenc}
\makeatletter
\newcommand*{\math@version@bold}{bold}
\DeclareMathOperator\DD{
  \textrm{%
    \usefont{T2A}{cmr}{\ifx\math@version\math@version@bold bx\else m\fi}{n}%
    \CYRD
  }%
}
\makeatother
\def\R{\ensuremath{\mathbb R}}
\def\T{\ensuremath{\mathbb T}}
\def\N{\ensuremath{\mathbb N}}

\def\Z{\ensuremath{\mathbb Z}}
\def\I{\ensuremath{{\bf 1}}}

\def\C{\ensuremath{\mathcal C}}

\def\A{\ensuremath{A^{(q)}}}

\def\X{\mathcal{X}}

\def\dist{\ensuremath{\text{dist}}}

\numberwithin{equation}{section}

\begin{document}

\title[\hfill\protect\parbox{0.975\linewidth}{Extremes  and extremal indices for level set  observables  on hyperbolic\\ systems.}]{Extremes  and extremal indices for level set  observables  on hyperbolic systems.}

\author[M. Carney]{Meagan Carney}
\address{Meagan Carney\\ Max Planck Institute for Physics of Complex Systems\\
Nonlinear Dynamics and Time Series Analysis\\
Dresden\\
Germany} \email{meagan@pks.mpg.de}

\author[M. Holland]{Mark Holland}
\address{Mark Holland\\ Department of Mathematics (CEMPS)\\
Harrison Building (327)\\
North Park Road\\
Exeter, EX4 4QF\\ UK} \email{M.P.Holland@exeter.ac.uk}
\urladdr{http://empslocal.ex.ac.uk/people/staff/mph204/}

\author[M. Nicol]{Matthew Nicol}
\address{Matthew Nicol\\ Department of Mathematics\\
University of Houston\\
Houston\\
TX 77204\\
USA} \email{nicol@math.uh.edu}
\urladdr{http://www.math.uh.edu/~nicol/}

\thanks{MC and MN were supported in part by NSF Grant DMS 1600780. MH acknowledges support of  EPSRC grant: EP/P034489/. MC thanks the Max Planck Institute-PKS, Dresden, for their hospitality while part of this work was completed.}

\date{\today}



\begin{abstract}
Consider an ergodic  measure preserving dynamical system $(T,X,\mu)$, and  an observable 
$\phi:X\to\mathbb{R}$. For the time series $X_n(x)=\phi(T^{n}(x))$, we establish limit
laws for the maximum process $M_n=\max_{k\leq n}X_k$ in the case where $\phi$ is an 
observable maximized on a curve or submanifold, and $(T,X,\mu)$ is a hyperbolic dynamical system.

Such observables arise naturally in weather and climate applications.
We consider the extreme value laws and extremal indices for these observables on
  hyperbolic toral automorphisms, Sinai dispersing billiards and coupled expanding maps. In particular 
  we obtain clustering and nontrivial extremal indices due to self intersection of submanifolds
  under iteration by the dynamics, not arising from any periodicity.

\end{abstract}

\maketitle

\section{Introduction}

 Suppose we have a time-series $(X_n)$ of real-valued random variables defined on a probability space $(X,\mu)$ and let $M_n:=\max \{X_1,\ldots, X_n\}$ be the sequence of successive maxima of $(X_i)$. There is a well-developed theory for these maximum values in the setting of $(X_n)$ i.i.d \cite{Embrechts, Galambos}.  If we consider a dynamical system $(T,X,\mu)$ such that $T\colon X\rightarrow X$ and an observable $\phi: X\to \R$, we can define a stochastic process by 
\[
X_n = \phi\circ T^n (x)
\]
for $x\in X$. In the case of modeling deterministic physical phenomenon, $T$ is usually taken as an ergodic, measure-preserving transformation, $\mu$ a probability measure and $\phi$ is a function with  some regularity, for example (locally) H\"{o}lder~\cite{V.et.al}. In extreme value literature, it is typically assumed that $\phi$ is a function of the distance $d(x,p)$ to a distinguished point $p$ for some metric $d$ so that $\phi (x)=f(d(x,p))$ for $x\in X$, and $f$ is a monotone decreasing function $f:(0,\infty)\to\mathbb{R}$. In this instance
$\sup_{x\in X}\phi(x)=\lim_{x\to p}\phi(x)$, and hence the set $\{\phi(x)\geq u\}$ corresponds to a neighborhood about  $p$. We shall refer to the set of all points $x\in X$ for which $\phi(x)$ achieves its maximum (with $\sup_{x\in X}\phi(x)=\infty$ allowed) as the extremal set $\mathcal{S}$. For convenience (and almost by convention) the observation $\phi (x)=-\log d(x,p)$ is often used, but scaling relations translate extreme value results for one functional form to another quite easily provided the extremal set of $\phi$ is unchanged. If the observable $\phi (x)=-\log d(x,p)$ is changed to another function of $d(x,p)$, then $\mathcal{S}$ remains equal to $\{p\}$. 
However, if the underlying extremal set $\mathcal{S}$ is changed, e.g. going from a point to a curve, then the proofs of extreme value results and the results themselves do not translate and new approaches are required. Indeed, even if the extremal set changes from one point to another, then the  extreme value laws may change (e.g. $p$ periodic versus $p$ non-periodic give different distributional extreme value laws)~\cite{Dichotomy,Ferguson_Pollicott,FFT3,Keller}.

Since the value of the  function $\phi\circ T^n (x)$ is larger the closer $T^n (x)$ is to the extremal set $\mathcal{S}$, there is a close relation between extreme value statistics for the time series $X_n=\phi\circ T^n(x)$ and return-time statistics to nested sets about $\mathcal{S}$~\cite{CC,Collet,DGS,FFT1,FFT3,Gupta,Hirata}. We focus on extreme value theory in this 
paper but it would be possible, though computationally very difficult,  to derive return time distributions which are simple Poisson (in the 
cases in which the extremal index is $\theta=1$) and  compound Poisson (in the cases in which the extremal
index $\theta<1$). The parameters in the compound Poisson distribution would in particular be difficult to compute but this would constitute an interesting investigation.
We expect this work could be carried out using basically the same toolkit from extreme value theory.
These parameters are calculated for  functions maximized at  periodic orbits in the setting of a hyperbolic toral automorphism~\cite{Dichotomy} 
and in~\cite{CNZ} for  functions maximized at  periodic orbits in Sinai dispersing billiard systems.  We discuss the concept of extremal index below, it is a number
$0 \le \theta \le 1$ which roughly quantifies the clustering of exceedances. We will say $\theta=1$ is a trivial
extremal index and $\theta<1$ a nontrivial extremal index.  For results along these lines see~\cite{Dichotomy,FFT3,FHN}.

Recent literature has focused on the case where the extremal set $\mathcal{S}$ is a single point $\{p\}$. In this paper we address some scenarios of interest where the observable is maximized on sets other than unique points in phase space, and in turn
describe how the extreme value law depends on the geometry of $\mathcal{S}$. We also describe a dynamical mechanism giving rise to a nontrivial extremal index which is not due to  periodicity. The recent preprint~\cite{Haydn_Vaienti} provides a different and axiomatic
approach to determining the limit laws (especially simple and compound Poisson distributions) for entry times into neighborhoods of sets of measure zero in dynamical systems. They present similar results to this paper on coupled map lattices and consider other dynamical and statistical examples, including some systems with polynomial decay of correlations.  We address here cases that are not easily
captured by axiomatic approaches. This happens for example, if the extremal set $\mathcal{S}$ fails certain transversality assumptions 
relative to the local (or global) stable and unstable manifolds of the system. We discuss these situations further in Sections 
\ref{sec.anosov} and \ref{sec.discussion}.

\subsection{Background on extremes for dynamical systems}\label{sec.background}
Suppose $(X_n)$ is a  stationary process  with probability distribution function $F_X(u):=\mu(X\leq u).$ We define an extreme value law (EVL) in the following way. Given $\tau\in\mathbb{R}$, let $u_n(\tau)$ be a sequence satisfying $n\mu(X_0>u_{n}(\tau))\to\tau$, as $n\to\infty$. We say that $(X_n)$ satisfies an extreme value law if 
\begin{equation}\label{eq.ev-law}
\mu(M_n\leq u_n(\tau))\to e^{-\theta\tau}
\end{equation}
for some $\theta\in(0,1]$. Here, $\theta$ is called the extremal index and $\frac{1}{\theta}$ roughly measures the average number of exceedances in a time window 
given that one exceedance has occurred.  When $(X_n)$ is  i.i.d. and   has a regularly varying tail it can be shown that this limit exists and $\theta=1$.

In the dependent setting for stationary $(X_n)$ the existence of an EVL has been shown provided dependence conditions $D(u_n)$ (mixing condition) and $D'(u_n)$ (recurrence condition) or similar conditions hold for the system~\cite{LLR,FFT1}.  Freitas et al~\cite{FnF}, based on Collet's work,  in turn gave  a condition $D_2(u_n)$ which has the full force of $D(u_n)$  in that together with $D'(u_n)$ it  implies the existence of an EVL and is easier to check in the dynamical setting. We describe 
more precisely these three conditions below.


There are, however, no general techniques for proving conditions $D_2(u_n)$ and $D'(u_n)$ and checking the latter is usually hard. $D'(u_n)$ is a short returns condition that is not implied by
an exponential decay of correlations. However $D_2(u_n)$ often follows from a suitable rate of decay of correlations.
Collet~\cite{Collet}  used the rate of decay of correlation of H\"{o}lder observations to establish $D(u_n)$ for certain one-dimensional 
non-uniformly expanding  maps.
Condition $D_2(u_n)$  is  easier to establish in the dynamical setting by estimating the  rate of decay of correlations of H\"older continuous observables 
or those of bounded variation and in practice is easier to verify.

For completeness we  now state conditions $D(u_n)$, $D_2(u_n)$ and $D^{'}(u_n)$.  If $\{X_n\}$ is a stochastic process define 
\[
M_{j,l} := \max\{X_j, X_{j+1}, \dots, X_{j+l}\}.
\]
 We will often write $M_{0,n}$ as $M_n$. We write $F_{i_1,\ldots,i_n} (u)$ for the joint distribution $F_{i_1,\ldots,i_n} (u)=\mu (X_{i_1} \le u, X_{i_2} \le u,\ldots, 
 X_{i_n} \le u)$.

\noindent {\bf Condition $D (u_n)$~\cite{LLR}} We say condition $D(u_n)$ holds for the sequence $X_0,X_1,\ldots, $ if for any integers
$i_1<i_2<\ldots < i_p< j_1<j_2<\ldots < j_{p'}\le n$, for which $j_1-i_p>t$
we have 
\[
|F_{i_1,i_2,\ldots, i_p, j_1,j_2,\ldots, j_{p'}} (u_n) -F_{i_1,i_2,\ldots, i_p} (u_n) F_{ j_1,j_2,\ldots, j_{p'}} (u_n)| \le \gamma(n,t)
\]
where $\gamma (n,t)$ is non-increasing in $t$ for each $n$ and $n\gamma(n,t_n)\to 0$ as $n\to \infty$ for some sequence $t_n=o(n)$, $t_n\rightarrow \infty$.

\noindent {\bf Condition $D_2 (u_n)$~\cite{FnF}} We say condition $D_2(u_n)$ holds for the sequence $X_0,X_1,\ldots, $ if for any integers $l$,$t$ and $n$
\[
|\mu ( X_0 >u_n, M_{t,l}  \le u_n )-\mu (X_0 >u_n)\mu ( M_{l} \le u_n)| \le \gamma(n,t)
\]
where $\gamma (n,t)$ is non-increasing in $t$ for each $n$ and $n\gamma(n,t_n)\to 0$ as $n\to \infty$ for some sequence $t_n=o(n)$, $t_n\rightarrow \infty$.

\noindent {\bf Condition $D^{'} (u_n)$~\cite{LLR}} We say condition $D^{'}(u_n)$ holds for the sequence $X_0,X_1,\ldots, $ if 
\begin{equation}\label{cond:dprime}
\lim_{k\to \infty}\limsup_n n\sum_{j=1}^{[n/k]}\mu(X_0>u_n,X_j>u_n)=0.
\end{equation}
Condition $D^{'} (u_n)$ controls the measure of the set of points of $(X_0>u_n)$ which return to the set 
relatively quickly, and is a condition that rules out  ``short returns''.  It is not a consequence of exponential decay of correlations and usually dynamical and 
geometric arguments are needed to verify Condition $D^{'} (u_n)$ in specific cases.

 In the dynamical case if the time series of observations $X_n=\phi\circ T^n$  satisfy $D(u_n)$ (or  $D_2(u_n)$) and $D'(u_n)$ 
(or some variation thereof) then an EVL holds.  In these results, we have extremal index $\theta=1$ for observables of the form $\phi (x)=f(d(x,p))$, maximized at generic $p\in X$ provided $p$ is non-periodic~\cite{Dichotomy,Ferguson_Pollicott,FFT3,FHN,HNT0,Keller}. For periodic $p$, EVLs have been derived for these systems with index $\theta< 1$~\cite{CNZ,Dichotomy,Ferguson_Pollicott,FFT3,Keller,V.et.al}.

For statistical estimation and fitting schemes such as block maxima or peak over thresholds methods \cite{Embrechts}, it is desirable to get a limit along linear sequences of the form $u_n(y)=y/a_n+b_n$. Here the emphasis is changed and the sequence $u_n(y)$ is now required to be linear in $y$. For example suppose $\phi(x)=-\log x$ is 
an observable on the doubling map of the interval $[0,1]$, $Tx=(2x)$ mod $1$, which preserves Lebesgue measure $\mu$. The condition $n\mu (\phi > u_n(y))=y$ implies
$u_n(y)=\log n-\log y$. Furthermore we know that $n\mu (\phi > u_n(y))=y$ implies $\mu (M_n \le u_n(y))\to e^{-y}$.  This is a nonlinear scaling. 
If we change variables to $Y=-\log y$ we obtain $n\mu (-\log x > Y+\log n)\to e^{-y}=e^{-e^{-Y}}$, a Gumbel law.

In general if we restrict to linear scalings $y\in\mathbb{R}$, we obtain a limit $n\mu (X_0>\frac{y}{a_n}+b_n)\to h(y)$ and hence
\[
\mu(a_n(M_n-b_n)\le y)\to e^{-h(y)} =G(y),\quad(n\to\infty).
\]
For i.i.d processes, if $G$ exists and is non-degenerate, then it takes three distinct forms $G(y)=e^{-h(y)}$ with either:
\begin{itemize}
\item[(i)] $h(y)=e^{-y}$, $y\in\mathbb{R}$ (Gumbel);
\item[(ii)] $h(y)=y^{-\alpha}$, $y>0$ and some $\alpha>0$ (Fr\'echet); 
\item[(iii)] $h(y)=(-y)^{\alpha}$, $y<0$ and some $\alpha>0$ (Weibull).
\end{itemize}
These three forms can be combined into a unified \emph{generalized extreme
value} (GEV) distribution (up to scale and location $u\to \frac{u-\alpha}{\sigma}$):
\begin{equation}\label{eq.gevlimit}
G_{\xi}(y)=
\begin{cases}
\exp\{-(1+\xi y)^{-\frac{1}{\xi}}\},\text{ if $\xi\neq 0$};\\
\exp\{-e^{-y}\},\text{ if $\xi= 0$}.
\end{cases}
\end{equation}
The case $\xi=0$ corresponds to the Gumbel distribution, $\xi>0$ corresponds to a Fr\'echet distribution, 
while $\xi<0$ corresponds to a Weibull distribution.  


Numerical fitting schemes for the GEV distribution are renormalized under place and scale transformations so that the extremal index (EI) is $\theta=1$~\cite[Theorem 5.2]{Coles}. Although it is theoretically possible to recover the EI by considering it as a function of these transformations, estimates in this way would have an undetectable level of error. Techniques to directly compute the EI, referred to as  \textit{blocks} and \textit{runs} estimators, have been proposed \cite[Section 3.4]{V.et.al}. Both methods utilize the definition of the EI (outlined above) by numerically estimating the ratio of the number of exceedances in a cluster to the total number of exceedances. Where these differ is in their definitions of a cluster; the runs estimator splits the data into fixed blocks of size $k_n$ so that a cluster is defined by the number of exceedances inside each fixed block while the blocks estimator introduces a run length of $q_n$ so that any two exceedances separated by a time gap of less than $q_n$ belongs to the same cluster. The problem with using these estimators in practical applications is their heavy dependence on the choice  of
the sequences  $k_n$ and $q_n$, respectively. 

Recent literature has provided more robust estimates of the extremal index. In particular the S\"{u}veges estimator \cite{suveges} has become more common in extreme value statistics \cite{sandro_coupled,V.et.al}. For a sequence $(X_j)$ $j = 1,\dots,n$ of random variables, let $q$ denote a fixed quantile and $l$ the location of exceedances $\{l:X_l>q\}$ above $q$. We define $T_i = l_{i+1}-l_i$ for $i = 1,\dots,N-1$ as the length of time between each consecutive recurrence. Let $S_i = T_i-1$ and $N_c = \sum_{i=1}^{N-1} {1}_{S_i \ne 0}$, so that $N_c$ is the number of clusters found by counting the set of recurrences separated by a time gap of at least length 1. Then the S\"{u}veges estimator of the extremal index given by,
\[
\hat{\theta} = \frac{\sum_{i=1}^{N-1} qS_i+N-1+N_c-[(\sum_{i=1}^{N-1} qS_i +N-1+N_c)^2-8N_c\sum_{i=1}^{N-1}qS_i]^{1/2}}{2\sum_{i=1}^{N-1}qS_i},
\]
can be viewed as the maximum likelihood estimator for the expected value of the number of recurrences coming from a point process defined by the compound Poisson distribution. We use this method to estimate the EI of the coupled map and the hyperbolic toral automorphism of  Section~\ref{sec.numerics}.



For dynamical systems, the corresponding problem of finding scaling constants  $a_n, b_n$ depends on both the regularity of 
$\mu$ and that of the observable $\phi(x)=f(d(x,p))$ in the vicinity of the  point $p$. 

For more general dynamical systems, these scaling relations depend on how the invariant measure scales on sets that shrink
to $p$. This problem has been addressed in the case where $\mu$ admits a smooth or regularly varying density function $h$.
 However, for general measures (such as Sinai Ruelle Bowen measures) and general observables, estimating $\mu(X>y/a_n+b_n)$ becomes more delicate, see \cite{FFT2, HVRSB}. However, an extreme law can still be obtained along some non-linear sequence $u_n(y)$, with bounds on the growth of $u_n(y)$, see \cite{GHN}.

Furthermore, for deterministic dynamical systems the extremal index parameter $\theta$
may be nontrivial due to periodicity.  
For the doubling map discussed above, if $p$ is a periodic point then $\theta=1-\frac{1}{2^q}$ where
$q$ is the period of the period point (see \cite{FFT3,Keller}).

In this article, we consider cases where $\phi$ is maximized on a more general extremal sets $\mathcal{S}$. For general $\mathcal{S}$ we cannot rely on previous methods adapted to observables of the form $\phi=f(d(x,p))$.

\subsection{Physical and energy-like observables.}\label{sec.energy}

In the study of extreme events in dynamical systems, having in mind applications to weather and climate modeling, the notion of a \emph{physical observable} was introduced and described in \cite{HVRSB, LFWK, SHRBV}. By physical observables we mean those of  form $\phi(x)=x\cdot v$ or $\phi(x)=x\cdot Ax$, where $A$
is $d\times d$ matrix, and $v$ a specified vector in $\mathbb{R}^d$. The former observable has planar level sets, while the latter has ellipsoidal level sets. In weather applications,
these observables correspond to measuring (respectively) the momentum and kinetic energy of the system. The level geometries of 
$\phi$ introduced additional technicalities in establishing extreme laws  relative to the cases where
the level sets are metric balls. These issues are discussed in detail in \cite{HVRSB}, where $\mathcal{S}$ had a complicated geometry but its
intersection with the attractor of the system was still a single point.
 In this article,
we mainly consider energy-like observables for which the extremal set $\mathcal{S}$ is achieved on a line segment  or submanifold.
We also discuss other extremal sets in Section \ref{sec.discussion}.

\subsection{Organization of the paper.} In Section \ref{sec.statement} we describe our main results on: hyperbolic toral automorphisms,
Sinai dispersing billiard maps, and coupled uniformly expanding maps. We calculate the extreme value distribution,
the extremal index and in some cases describe briefly the Poisson return time process. In particular we describe a method for
obtaining a nontrivial extremal index which is not due to periodic behavior but rather self-intersection of a set of non-periodic points under the 
dynamics. 
  Beyond existing approaches, we have to develop arguments that deal with both
the geometry of $\mathcal{S}$, and the recurrence properties of the dynamical systems under consideration. In our examples the underlying invariant measures have regular
densities with respect to Lebesgue measure. This enables us to obtain analytic results
on the GEV parameters and the extremal index. We also compare our results to numerical schemes, see Section \ref{sec.numerics}. 
We conclude with a discussion~\ref{sec.discussion} on how  the methods we have developed might be applied to general observables
whose extremal sets have more complicated geometries. 

\section{Statement of Results}\label{sec.statement}

\subsection{Hyperbolic toral  diffeomorphisms}\label{sec.anosov}
 We consider hyperbolic toral automorphisms of the two-dimensional torus $\T^2$ induced by a matrix
\[T= \left( \begin{array}{cc}
a  & b \\
c & d  
 \end{array} \right).\] 
 with integer entries, $\det(T)=\pm 1$ and no eigenvalues on the unit circle. We will assume that both eigenvalues are positive in what follows  to simplify the discussion and proofs. Such maps preserve Haar measure $\mu$ on $\T$.
  A well-known example is the Arnold Cat map
  \[ \left( \begin{array}{cc}
2  & 1 \\
1 & 1 
 \end{array} \right).\]

 We consider $\T^2$ as the unit square with usual identifications with universal cover $\R^2$. $T$ preserves the Haar measure $\mu$ on $\T^2$ and 
 has exponential decay of correlations for Lipschitz functions, in the sense that there exists $\Lambda \in (0,1)$ such that
 \[
 |\int \phi \circ T^n \psi d\mu -\int \phi d\mu \int \psi d\mu |\le C\|\phi\|_{Lip} \|\psi \|_{Lip} \Lambda^n
 \]
 where $C$ is a constant independent of $\phi$, $\psi$ and $\|.\|_{Lip}$ is the Lipschitz norm~\cite{Bowen}.

 For a set $D$, we define 
$d_{H}(x,D)=\inf\{d(x,y):y\in D\}$ (for Hausdorff distance) , where $d$ is the distance in ambient (usually Euclidean) metric. $\overline{D}$ denotes 
the closure of $D$ and we define  $D_{\epsilon}=\{x:d_{H}(x,\overline{D})\leq\epsilon\}$ is an $\epsilon$ neighborhood of $D$. 
As outlined in Section~\ref{sec.energy}, the observables we consider take the form  $\phi (x)=f(d_H(x,L))$ where $x=(x_1,x_2)\in \T^2$ and $L\subset \T$ is a line segment with direction vector $\hat{L}$ and finite length $l(L)$. The function $f:[0,\infty)\to\mathbb{R}$ is a
smooth monotone decreasing function. We will take $f(u)=-\log (u)$. To fix notations, we also need to later consider $\epsilon$-tubes around $\mathcal{S}$. Thus if $\mathcal{S}$ is a line, or curve,
and $\epsilon$ is small, then $\mathcal{S}_{\epsilon}$ is a thin tube.


 The matrix $DT$ has two unit  eigenvectors $v^{+}$ and $v^{-}$ corresponding to the respective eigenvalues 
$\lambda_{+}=\lambda>1$, and $\lambda_{-}=\lambda^{-1}<1$.
 We can write $\hat{L}=\alpha v^{+} + \beta v^{-}$ for some coefficients $\alpha$, $\beta$ and so $DT^n \hat{L}=\alpha \lambda_{+}^n v^{+}+  \beta \lambda_{-}^n v^{-}$.
If we let $v^{(n)}$ denote a unit vector in the direction of $DT^n \hat{L}$  and $\alpha\not =0$, $\beta = 0$ then $ v^{(n)}$ aligns with the direction $v^+$ as $n\rightarrow \infty$


 If $L$ is aligned with the unstable direction, we may lift $L$ to $\hat{L}$ on a fundamental domain of the cover $\R^2$ of $\T^2$ and write $\hat{L}=\hat{p}_1+ tv^+$, $t\in [0,l(L)]$, $\hat{p}_1\in \R^2$. Thus  $L=\pi (\hat{p}_1+ tv^+)$, $t\in [0,l(L)]$ where $\pi: \R^2\to \T^2$ is the usual projection $\pi : \R^2\to \R^2/\Z^2$. We write the 
endpoint of $\hat{L}$ as $\hat{p}_2$, i.e. $\hat{p}_2=\hat{p}_1+l(L)v^+$. We will also identify the vectors $\pi \hat{p}_1$  and $\pi \hat{p}_2$ with the corresponding points 
 $p_1$ and $p_2$ in $\T^2$. Similarly if $L$ is aligned with the stable direction, we may lift $L$ to $\hat{L}$ on a fundamental domain of the cover $\R^2$ of $\T^2$ and write $L=\pi (\hat{p}_1+ tv^{-})$, $t\in [0,l(L)]$, $\hat{p}_1\in \R^2$. 
Again we write the 
 endpoint of $\hat{L}$ as $\hat{p}_2$, i.e. $\hat{p}_2=\hat{p}_1+l(L)v^-$. We will also identify the vectors $\hat{p}_1$  and $\hat{p}_2$ with the corresponding points 
they project to under $\pi$, written $p_1$ and $p_2$.

\begin{theorem}\label{thm.anosov} Let $T: \T^2 \to \T^2$  be a hyperbolic toral automorphism with positive eigenvalues
$\lambda^+=\lambda>1$, $\lambda^{-}=\frac{1}{\lambda}<1$. Let $\mu$ denote Haar measure on $\T^2$.  Let $L\subset \T^2$ be
 the projection $\R^2\to \T^2$  of a line segment $\hat{L}$ with  finite length $l(L)$. 
Define  $\phi (x)=-\log (d_H(x,L))$, $\phi:\T^2 \to \R$. Define $M_n (x) =\max \{ \phi (x), \phi (Tx), \ldots, \phi (T^{n-1} (x))\}$.
Then \begin{equation}\label{eq.maxlimit1-thm}
\lim_{n\to\infty}\mu(M_n\leq y+\log n+l(L))= \exp\{-\theta e^{-y}\}.
\end{equation}
where the extremal index $\theta$ is determined by these cases. If:
\begin{enumerate}
\item  $L$ is not aligned with the stable $v^{-}$  or unstable  $v^{+}$ direction then $\theta=1$.
\item   $L$ is aligned with the unstable direction $v^{+}$ and $\pi (\hat{p}_1+tv^{+})$, $-\infty < t < \infty$ contains
no periodic  points then $\theta=1$.
\item If $L$ is aligned with the stable direction $v^{-}$  and $\pi (\hat{p}_1+tv^{-})$, $-\infty < t < \infty$ contains
no periodic points then $\theta=1$.
\item $L$ is aligned with the stable $v^{-}$ or unstable $v^{+}$ direction and $L$ contains a periodic point of 
prime period $q$ then $\theta=1-\lambda^{-q}$.
\item\label{extremal-range1}   $L$ is aligned with the unstable direction $v^{+}$, $L$ contains no periodic 
points  but $\pi (\hat{p}_1+tv^{+})$, $-\infty < t < \infty$ contains a periodic point $\zeta$ of prime period $q$; then 
$L \cap T^qL=\emptyset$ implies $\theta=1$; otherwise if  $ L \cap T^qL\neq \emptyset$ then 
$(1-\lambda^{-q}) \le \theta \le 1$ and all values of $\theta$ in this range can be realized depending on the length and placement of $L$;
\item\label{extremal-range2} $L$ is aligned with the stable direction $v^{-}$, $L$ contains no periodic 
points but $\pi (\hat{p}_1+tv^{-})$, $-\infty < t < \infty$ contains a periodic point of prime period $q$;
then $L \cap T^qL=\emptyset$ implies $\theta=1$; otherwise if  $ L \cap T^qL\neq \emptyset$ then 
$(1-\lambda^{-q}) \le \theta \le 1$ and all values of $\theta$ in this range can be realized depending on the length and placement of $L$.

\end{enumerate}
\end{theorem}

\begin{remark}
For cases \ref{extremal-range1}, and \ref{extremal-range2} we  may 
realize any value of $\theta$ in the range $[(1-\lambda^{-q}),1]$. This will be demonstrated in the proof, where the value of $\theta$ is given as  
a function of  the locations of $\pi \hat{p}_1$ and $\pi \hat{p}_2$ relative to the period-$q$ point on  a  continuation of $L$. This formula is difficult to 
state in an elegant way in full generality.
\end{remark}

\begin{remark}
 In Theorem \ref{thm.anosov} we have focused on
the particular case $f(u)=-\log u$ which gives rise to a Gumbel distribution. For other
functional forms, such as $f(u)=u^{-\alpha}$, $(\alpha>0)$  we obtain
corresponding limit laws. 
\end{remark}

\begin{remark}
Since all  periodic points of $T$ have rational coefficients  $(\frac{p_1}{q_1}, \frac{p_2}{q_2})$ and $v^+$,$v^-$ have irrational slopes it follows that if $\pi (\hat{p}_1+tv^{+})$, $-\infty <t <\infty$ contains a periodic point it contains at most one, and similarly for $\pi (\hat{p}_1+tv^{-})$, $-\infty <t <\infty$.

\end{remark}

Using exponential decay of correlations 
of the map, we show that for small $\epsilon$-tubes $L_{\epsilon}$
around $L$, we have (for all $j$ sufficiently large)
$\mu (T^jL_{\epsilon}\cap 
L_{\epsilon}) \leq C\mu (L_{\epsilon})^2.$ This enables
us to easily verify the form of the $D(u_n)$ condition of Leadbetter et al~\cite{LLR} that we use.

The argument in the case  that $L$ is aligned with $v^{+}$ turns out to be  the most subtle. We need a detailed analysis of how the forward images $T^j L$
wrap around the torus. It is clear that these forward images are dense, but
we need quantitative information on how quickly these images
become uniformly distributed.  Such considerations are not necessary in the case 
where $\mathcal{S}$ is a single point, e.g. as discussed in \cite{Dichotomy}, and furthermore
this scenario is not easily captured by axiomatic approaches, as discussed in \cite{CC, Haydn_Vaienti}.
The close alignment of $\mathcal{S}$ with the unstable manifold appears non-generic in this hyperbolic toral automorphism example. 
However, for general observables one could imagine level set geometries failing transversality conditions generically,
e.g. if $\{\phi>u_n\}$ has a non-trivial boundary, which perhaps coils or accumulates upon itself. These scenarios would have
to be treated on a case by case basis. 


\subsection{Sinai dispersing billiards maps}\label{sec.billiards}

We now consider another setting in which it is natural to have a smooth observable maximized on a line segment.  
Suppose  $\Gamma = \set{\Gamma_i, i = 1:k}$ is  a family of pairwise disjoint, simply connected $C^3$ curves with strictly positive curvature
on the two-dimensional torus $\mathbb{T}^2$.  The billiard flow  $B_t$ is the dynamical system 
generated  by the motion of  a point particle in $Q= \mathbb{T}^2/(\cup_{i=1}^k (\mbox{ convex hull of } \Gamma_i))$ which moves with  constant unit velocity inside $Q$
until it hits $\Gamma$, then it undergoes an elastic collision where angle of incidence equals angle of reflection.
If each $\Gamma_i$ is a circle and the system is lifted periodically to $\R^2$ then this system is called a periodic Lorentz gas and was a model in  the pioneering work of Lorentz on electron motion in  conductors. 


It is often easier to consider the billiard map  $T: \partial Q \to \partial Q$, derive statistical properties for it and then 
deduce corresponding properties for the flow. In this paper we will focus on limit laws for the billiard map.  Let $r$ be   the natural one-dimensional coordinate of 
$\Gamma$ given by  arc-length and let $n(r)$ be the outward normal to $\Gamma$ at the point $r$. For each
$r\in \Gamma$ the tangent
space at $r$ consists of unit vectors $v$ such that $(n(r),v)\ge 0$. We identify such a  unit vector $v$ with an 
angle $\vartheta \in [-\pi/2, \pi/2]$. The phase space  $M$ is then parametrized by $M:=\partial Q=\Gamma\times [-\pi/2, \pi/2]$, and $M$ consists of  the 
points $(r,\vartheta)$. $T:M\to M$ is the Poincar\'e map that gives the position and angle $T(r,\vartheta)=(r_1,\vartheta_1)$  after a point $(r,\vartheta)$
flows under $B_t$ and collides again with $M$, according to  the rule  angle of incidence equals angle of reflection. The billiard map preserves a measure $d\mu=c_{M} \cos \vartheta dr d\vartheta$ equivalent to $2$-dimensional Lebesgue measure $dm=drd\vartheta$ with
density $\rho (x)=c_M \cos \vartheta$ where $x=(r,\vartheta)$. 

For this class of billiards the stable and unstable foliations lie in strict cones $C^u$ and $C^s$ in that 
the graphs $\vartheta(r)$ of local unstable manifolds have uniform bounds on the slopes of their tangent vectors which lie in the cone $C^u$,  $s_0 \le \frac{d\vartheta}{dr} \le s_1$, and similarly tangents to  local stable manifolds
lie in a cone $C^s$, $- t_1 \le \frac{d\vartheta}{dr} \le -t_0$, for some strictly positive constants  $s_0,t_0,s_1,t_1$.

We will assume a line segment  $L$ with direction vector $\hat{L}$ is uniformly transverse to
$C^s$ and $C^u$.
More precisely, we will consider functions maximized on line segments $L=\{x=(r,\vartheta): x\cdot v=c\}$,  $v=(v_1,v_2)$, which are transverse to the stable and unstable cone of  directions. For example the line segment $r=r_0$, which
is a position on the table rather than the  point $(r_0,\vartheta_0)$ (which is in phase space). We note that~\cite{Pene_Saussol}
 studied distributional and almost sure return time limit laws to the position $r=r_0$. In our setting the precise extreme law (Weibull, Fr\'echet or Gumbell)
depends upon the observable we choose but results may be transformed from one observable to another in a standard way.  We will take the function
$\phi (r,\vartheta)=1-d_H (x,L)$ which because  it is bounded will lead to a Weibull distribution.  We assume the finite horizon condition, namely that the time of flight of the billiard flow between collisions is bounded above and also away from zero.
Under the  finite  horizon condition Young~\cite{Y98} proved that the  billiard map has exponential decay of correlations for H\"{o}lder observables. 
A good reference for background results for this section are the papers~\cite{BSC1,BSC2,CM07,Y98} and the book~\cite{CM}. 

Let $L$ be a line segment  transverse to the stable and unstable cones  and $\phi (r,\vartheta)=1-d_H(x,L)$. Let 
$y>0$. We define a sequence $u_n(y)=y/a_n+b_n$ by the requirement $n \mu \{ \phi > u_n(y)\}=y$. Apart from complication arising from the 
invariant measure having a cosine term, $a_n$ scales like $\frac{1}{n}$.
The set $\{ \phi > u_n\}$ is a rectangle $U_n$ with center $L$  roughly of width $\frac{Cy}{ n}$ for some constant $C$. Note that we assume $L$ is not aligned in either the 
 unstable   or the stable direction, so the following result is expected from the hyperbolic toral automorphism case.

\begin{theorem}\label{thm:billiards}
Let $T:M\to M$ be a  planar dispersing  billiard map with invariant measure $d\mu=c_{M} \cos \vartheta dr d\vartheta$.  Suppose $x=(r,\theta)$ and $\phi (x)=1- d_H(x,L)$ where $\hat{L}$ is not in the
 unstable   cone $C^u$ or the stable cone $C^s$.
Let $M_n (x)=\max\{ \phi (x), \phi \circ T(x), \ldots, \phi \circ T^{n-1} (x)\}$.
Then $\mu (M_n \le u_n(y))\to e^{-y}$ as $n\to \infty$. In particular the extreme index $\theta=1$.

\end{theorem}

\begin{remark}
 We now make some remarks on what we conjecture in the case that a $C^2$ curve $L$ 
 is contained in, i.e. a piece of, a local unstable or local stable manifold and  $\phi (x)=1- d_H(x,L)$. If $L$ is part of a local unstable manifold and $T^n L$  has no  self-intersections with $L$ then the extremal index is one. The proofs we give in the case of the hyperbolic toral automorphism for this scenario break down but the techniques of the recent preprint~\cite{Fan_Yang} probably extend to this case. If $L$ contains a periodic point $\zeta$ of period $q$ then the extremal index would be roughly $\theta\sim 1-\frac{1}{|DT_u (\zeta)|^q}$ where $DT_u(\zeta)$ is the expansion in the unstable direction at $\zeta$ with a correctional factor due to the 
 conditional measure on the unstable manifold which contains $L$.  If $L$ does not contain a periodic point but its continuation in the unstable manifold does contain
 a periodic point of period $q$  then  as in case (5) of Theorem 2.1, if $T^qL\cap L=\emptyset$ then $\theta=1$, otherwise we expect $\theta$
 to lie roughly in the range $1-\frac{1}{|DT_u (\zeta)|^q} \le \theta \le 1$ (with all values of $\theta$ being realizable depending on the length and placement of $L$). If $L$ is part of a local stable manifold and $T^n L$  has no  self-intersections with $L$ then the extremal index $\theta =1$. If $L$ contains a periodic point $\zeta$ of period $q$ then the extremal index would be roughly $\theta\sim 1- |DT_s (\zeta)|^q$ where $DT_s(\zeta)$ is the expansion in the stable direction at $\zeta$. If $L$ does not contain a periodic point but its continuation in the unstable manifold does contain
 a periodic point of period $q$  then  as in case (6) of Theorem 2.1, if $T^qL\cap L=\emptyset$ then $\theta=1$, otherwise we expect $\theta$
 to lie roughly in the range $1- |DT_s (\zeta)|^q \le \theta \le 1$ (with all values of $\theta$ being realizable depending on the length and placement of $L$).

\end{remark}



\subsection{Coupled systems of uniformly expanding maps.}\label{sec.coupled}

Now we consider a simple class of coupled mixing expanding maps of the unit interval, similar to those examined in~\cite{sandro_coupled}. In fact we were
motivated by the comprehensive  work of~\cite{sandro_coupled} (which uses sophisticated transfer operator techniques) to develop in this paper  an alternate probabilistic approach  in a  coupled maps setting. The recent preprint~\cite{Haydn_Vaienti} presents similar results to ours in the case of returns to the diagonal
$\{ x_1=x_2=\ldots =x_n\}$. 
  Let $T$ be a $C^2$ uniformly expanding map of $S^1$  and suppose that $T$ has an invariant  measure $\mu$ with  density $h$ bounded above and below from zero. In ~\cite{sandro_coupled} piecewise
  $C^2$ expanding maps were considered but we will limit our discussion to smooth maps. 
 We use all-to-all coupling and first discuss the case of 
two coupled maps for clarity.

Let $0<\gamma<1$ and define
\begin{equation}\label{2maps}
F(x,y)= ((1-\gamma)Tx + \frac{\gamma}{2} (Tx+Ty),(1-\gamma)Ty + \frac{\gamma}{2} (Tx+Ty) )
\end{equation}
so that $F: \T^2\to \T^2$. We assume that $F$ has an an invariant measure $\mu$ on $\T^2$ with  density $\tilde{h}$ on $\T^2$ bounded 
above and also bounded below away from zero almost surely. We will require also that there exists $\epsilon>0$ and $0<\alpha \le 1$ such that
\[
|\tilde{h}|_{\alpha} :=\esssup_{0<\epsilon <\epsilon_0,x\in\T^2} \frac{1}{\epsilon^{\alpha}} \int {\it osc}(h,B_{\epsilon} (x))dm <\infty
\]
where ${\it osc}(h,A)=\esssup_{x\in A}-\essinf_{x\in A}$ for any measurable set $A$. The semi-norm $|.|_{\alpha}$ and this notion of regularity was described in~\cite{sandro_coupled}
and established in several of their examples.  An invariant  density for $F$ cannot reasonably be assumed to be continuous or Lipschitz. For example a slight perturbation of  the doubling map
of the unit circle $T(x)= (2x)$ (mod $1$) to the map $T(x)= ((2+\epsilon)x)$ (mod $1$) gives rise to a map with invariant density which is of bounded variation but not Lipschitz or even continuous.
$|.|_{\alpha}$ can be completed to a norm  $\|.\|_{{\it osc},\alpha}$ by defining $\|.\|_{{\it osc},\alpha}=|.|_{\alpha}+\|.\|_{1}$. The value of $\epsilon_0$ and $\alpha$ does not matter in our subsequent discussion. We note that the bounded variation norm and the quasi-H\"older norm $\|.\|_{{\it osc},\alpha}$ are particularly suited to handle
dynamical systems with discontinuities or singularities.

 We also assume a strong form of exponential decay of correlations in the sense that for 
all Lipschitz $\Phi$,  $L^{\infty}$ $\Psi$ on $\T^2$ there exists $C_1>0$ and $C_2>0$ such that for all $n$
\begin{equation}\label{mixing}
\Theta_n(\Phi,\Psi):=\left|\int \Phi\cdot\Psi \circ F^n d\mu-\int \Phi~d\mu\,\int \Psi~d\mu\right| \le C_1 e^{-C_2 n} 
\| \Phi\|_{\mathrm{Lip}} \|\Psi \|_{\mathrm{\infty}},
\end{equation}
where $\|\cdot\|_{\mathrm{Lip}}$ denotes the Lipschitz norm and $\|.\|_{\infty}$ denotes the $L^{\infty}$ norm. We note that this  assumption is not made for  (and does not hold for) hyperbolic toral automorphisms or Sinai dispersing billiards.

The function $\Theta_n(\Phi,\Psi)$ is called the correlation function.

Let $\phi (x,y)=-\log |x-y|$, a function maximized on the line segment (or circle) $L=\{(x,y): y=x\}$.  In this setting $L$ is invariant under $F$ and the orthogonal direction to 
$L$ is  uniformly repelling. Note that the projection of $(x,y)$ onto $L$ is the point $(\frac{x+y}{2},\frac{x+y}{2})$ and the projection on $L^{\perp}$ is $(x-\frac{x+y}{2},y-\frac{x+y}{2})$. Close to
$L$ we have uniform expansion away from $L$ in the $L^{\perp}$ direction under $F$. 
This is because  $y-x \mapsto (1-\gamma)[Ty-Tx]$ under $F$ so writing $y-x=\epsilon$ we see $\epsilon \rightarrow (1-\gamma)[T(x+\epsilon)-T x]\sim (1-\gamma) DT(x) \epsilon+O(\epsilon^2)$. There is uniform
 repulsion away from the invariant line $L$. This observation simplifies many of the geometric arguments we use 
 to establish extreme value laws.  
 
 In the more general case of $m$-coupled maps we define 
\begin{equation*}
F(x_1,x_2,\ldots, x_m):=\left(F_1(x_1,x_2,\ldots, x_m),\ldots, F_m(x_1,x_2,\ldots, x_m)\right),
\end{equation*}
with
\begin{equation}\label{m-map}
F_j(x_1,x_2,\ldots, x_m)
= (1-\gamma)T(x_j) + \frac{\gamma}{m} \sum_{k=1}^{m} T(x_k),
\end{equation}
for $j\in[1,\ldots, m]$. For these maps, we assume:
\begin{itemize}
\item[(A)] there exists a mixing invariant measure $\mu$ with density $\tilde{h}$, $\|\tilde{h}\|_{{\it osc},\alpha}<\infty$, on $\T^m$ bounded 
above and below away from zero;
\item[(B)] exponential mixing for Lipschitz functions versus $L^{\infty}$ functions as in Equation~\ref{mixing}.
\end{itemize}

\begin{remark}
Using the spectral analysis of the transfer operator of this system as in~\cite{sandro_coupled} and standard perturbation theory it can be shown that
(A) and (B) hold if $\gamma$ is sufficiently small as the uncoupled system is uniformly expanding.
\end{remark}

We consider a function maximized on $L=\{ (x_1,x_2,\ldots ,x_m): x_1=x_2=\ldots =x_m\}$.
The component of a point or vector $x=(x_1,x_2,\ldots, x_m)$ orthogonal to $L$ is 
$x^{\perp}=(x_1-\bar{x}, x_2-\bar{x}, \ldots, x_m -\bar{x})$ where $\bar{x}=\frac{1}{m}\sum_{j=1}^m x_j$. We define $\|(x_1,x_2,\ldots, x_m )\|=\max_j |x_j|$ and define for $x=(x_1,x_2,\ldots, x_m)$ 
\[
\phi (x)=-\log (\|x^{\perp}\|).
\]
The function $\phi$ is maximized on $L$, and large values of $\phi\circ F^n (x)$ indicate the orbit of $x$ is close to full synchrony of the 
coupled systems at time $n$.  Writing $p_i=x_i -\bar{x}$ we have $\sum_{i=1}^m p_i=0$.  Note if we have a vector $(\Delta p_1,
\Delta p_2,\ldots, \Delta p_m)$ orthogonal to $L$ we have $\sum_{i=1}^m \Delta p_i=0$. Thus in a sufficiently small neighborhood of $L$ we may write (for $j\in[1,\ldots,m]$)
$$F_j(x_1-\bar{x}, x_2-\bar{x},\ldots, x_m -\bar{x}) 
=(1-\gamma) DT\Delta p_j+\frac{\gamma}{m} \sum_{k=1}^m DT \Delta p_k +O(\max_k \Delta p_k)^2,  $$
\[
=(1-\gamma) DT\Delta p_j +O(\max_k \Delta p_k)^2
\]
where we have used
twice-differentiability and the fact that $\sum_{i=1}^m \Delta p_i=0$. Hence again there is uniform expansion in a sufficiently small neighborhood of $L$ in 
 the direction of the  $n-1$ dimensional subspace  orthogonal to $L$.

For $y>0$ 
define $u_n(y)$ by $n\mu (\phi >u_n (y))=y$, and $U_n=\{\phi >u_n (y)\}$.  
It can be seen that if $F$ is a map of $\T^m$ then $u_n\sim \frac{1}{m}[\log n -\log y]$, the precise relation depends
upon the density $\tilde{h}$ of the invariant measure. The precise functional form of $\phi$ is not important
as a different choice of $\phi$ would lead to a different scaling. 

\begin{theorem}\label{thm:coupled}

Let $F:\T^m\to \T^m$ be a  coupled system of expanding maps satisfying (A) and (B). Define $p^{\perp}=(x_1-\bar{x}, x_2-\bar{x}, \ldots, x_m -\bar{x})$ where $\bar{x}=\frac{1}{m}\sum_{j=1}^m x_j$ Suppose $\phi (p)=-\log (\|p^{\perp}\|)$.
Let $M_n (x)=\max\{ \phi (x), \phi \circ F(x), \ldots, \phi \circ F^{n-1} (x)\}$.
Then $\mu (M_n \le u_n(y))\to e^{-\theta y}$ as $n\to \infty$
where
\[
\theta=1-[\int_L \frac{1}{[(1-\gamma)DT(x)]^{m-1}} \tilde{h}(x)dx ].
\]

\end{theorem}

We may also consider blocks of synchronization, as in~\cite[Section 7.2]{sandro_coupled} where we take the observable maximized on a set $L$
consisting of  synchrony on  subsets of distinct lattice sites, for example of form $L=\{ (x_1,x_2,\ldots, x_m): x_{i_1}=x_{i_2}=\ldots =x_{i_k}, x_{j_1}=x_{j_2}=\ldots x_{j_l} \}$.
The main purpose of this section is to illustrate our geometric approach, so we will give one result of this type.

\begin{theorem}\label{thm:block}

Let $F:\T^m\to \T^m$ be a  coupled system of expanding maps satisfying $(A)$ and $(B)$.  Let $0<k\le m$ and choose 
$k$ distinct lattice sites $x_{i_1}$, $x_{i_2}$,$\ldots$, $x_{i_k}$. Define the subspace
$L=\{ (x_1,x_2,\ldots, x_m): x_{i_1}=x_{i_2}=\ldots x_{i_k}\}$ of dimension $m-k+1$ and $\bar{x}=\frac{1}{k}\sum_{j=1}^k x_{i_j}$.

Suppose $\phi (p)=-\log (\max_{j=1,\ldots ,k} |x_{i_j}-\bar{x}|)$.
Let $M_n (x)=\max\{ \phi (x), \phi \circ F(x), \ldots, \phi \circ F^{n-1} (x)\}$.
Then $\mu (M_n \le u_n(\tau))\to e^{-\theta\tau}$ as $n\to \infty$
where
\[
\theta=1-[\int_L \frac{1}{[(1-\gamma)DT(y)]^{k-1}} \tilde{h}(y)dy ]
\]
where $y$ is the  natural co-ordinatization of the $m-k+1$ dimensional subspace $L$.
\end{theorem}

\section{Extreme value  scheme of proof}\label{sec.extremeproof}

Our proofs are based on ideas from extreme value theory. We will use two conditions, adapted to the dynamical setting, introduced in the important work~\cite{FFT5}
that are based on $D(u_n)$ and $D_2 (u_n)$ but also allow a computation of the extremal index.

Let $X_n=\phi\circ T^n$ and define 
\[
A_n^{(q)} :=\lbrace X_0>u_n, X_1 \le u_n,\ldots, X_q\leq  u_n\rbrace
\]
For $s,l \in \mathbb{N}$ and a set $B\subset M$, let
\[
\mathscr{W}_{s,l}(B)=\bigcap_{i=s}^{s+l-1} T^{-i}(B^c).
\]
Next we describe the two conditions introduced in~\cite{FFT5}.\\

\noindent {\textbf{ Condition $\DD_q(u_n)$}}: We say that $\DD_q(u_n)$ holds for the sequence $X_0,X_1,\ldots$ if, for every  $\ell,t,n\in \mathbb{N}$
\[
\left|\mu\left(A_n^{(q)}\cap
  \mathscr{W}_{t,\ell}\left(A_n^{(q)}\right) \right)-\mu\left(A_n^{(q)}\right)
  \mu\left( \mathscr{W}_{0,\ell}\left(A_n^{(q)}\right)\right)\right|\leq \gamma(q,n,t),
\]
where $\gamma(q,n,t)$ is decreasing in $t$ and  there exists a sequence $(t_n)_{n\in \mathbb{N}}$ such that $t_n=o(n)$ and
$n\gamma(q,n,t_n)\to0$ when $n\rightarrow\infty$.\\

 We consider the sequence $(t_n)_{n\in\N}$ given by condition $\DD_q(u_n)$ and let $(k_n)_{n\in\N}$ be another sequence of integers such that as $n\to\infty$,
\[
k_n\to\infty\quad \mbox{and}\quad  k_n t_n = o(n).
\]

\noindent {\textbf{ Condition $\DD'_q(u_n)$}}:  We say that $\DD'_q(u_n)$
holds for the sequence $X_0,X_1,\ldots$ if there exists a sequence $(k_n)_{n\in\N}$ as above  and such that
\[
\lim_{n\rightarrow\infty}\,n\sum_{j=q+1}^{\lfloor n/k_n\rfloor}\mu\left( A_n^{(q)}\cap T^{-j}\left(A_n^{(q)}\right)
\right)=0.
\]

We note that, taking $U_n:=\{ X_0>u_n\}$ for $A_n^{(q)}$, which corresponds to non-periodic behavior, in condition $\DD'_q(u_n)$ corresponds to condition $D'(u_n)$ from~\cite{LLR}. We will abuse notation and consider
$ U_n:=\{ X_0>u_n\}$  as the  case of  $A_n^{(q)}$ with $q=0$. 

Now let
\[
\theta=\lim_{n\to\infty}\theta_n=\lim_{n\to\infty}\frac{\mu(A^{(q)}_n)}{\mu(U_n)}.
\]

\begin{remark} 
 In a dynamical setting verifying these two conditions picks up the main  underlying  periodicity or more generally recurrence properties of the system, for example 
returns to a periodic point of prime period $q$, and determines the extremal index. However, as we demonstrate,  other recurrent phenomena may give rise to an extremal index not equal to unity.
We show below that the self-intersection of a line segment  $L$, $T(L)\cap L\not = 0$ (none of whose points are periodic) may lead to a nontrivial  extremal index
for functions maximized on $L$. For a more detailed discussion of extremal indices see~\cite{FFT5}. 
\end{remark}

From \cite[Corollary~2.4]{FFT4}, it follows that to establish Theorem 2.1  it suffices to prove  conditions $\DD_q(u_n)$ and $\DD'_q(u_n)$ for $q=0$ in the 
non-recurrent case $\theta =1$ and for 
$q>0$ corresponding to the `period' of the  cases where there is some  recurrence phenomena ($\theta <1$).
In both cases
\[
\lim_{n\to\infty}\mu(M_n\leq u_n (y))= e^{-\theta y}.
\]

The scheme of the proof of Condition $\DD_q(u_n)$ is itself somewhat standard~\cite{Dichotomy,FHN} and is a consequence of suitable decay of 
correlation estimates. We outline it for completeness,
indicating the modifications that need to be made for the different geometries of $A^{(q)}_n$. 
The main work will be in establishing  Condition $\DD'_q(u_n)$.

\subsection{Proof of Theorem~\ref{thm.anosov}}

In the first instance we check condition $\DD_q(u_n)$. We recall some useful statistical properties of hyperbolic toral automorphisms.
In the case where $\Phi$ and $\Psi$ are Lipschitz continuous functions, it is known for hyperbolic toral automorphisms that there exists $C>0$, 
$\tau_0\in(0,1)$ such that
\begin{equation}\label{eq.cor-lip}
\left|\int \Phi (\Psi\circ T^n) d\mu-\int \Phi d\mu \int \Psi d\mu \right| \leq C\tau_{0}^n\|\Phi\|_{\mathrm{Lip}}\|\Psi\|_{\mathrm{Lip}},
\end{equation}
Furthermore if $\Psi$ is constant on local stable leaves corresponding to a Markov partition, 
then the Lipshitz norm of $\Psi$ on the right-hand side of equation (\ref{eq.cor-lip}) can be replaced by the
$L^{\infty}$ norm~\cite[Section 4]{Y98}. This fact will be useful when checking $\DD_q(u_n)$, see Proposition \ref{prop.ddq} in Section
\ref{sec.ddq} below.

Consider now a set $D$, whose boundary $\partial D$ is a union
of a finite number of smooth curves, so that $\mu(\partial D)=0$. Let $W^{s}_{1}(x)$ denote the local stable manifold through $x$. 
We define,
\begin{equation}\label{eq.Hk}
 H_{k}(D):= \left\{x \in D:T^{k}(W^s_1(x))\cap \partial D\ne\emptyset\right\}.
 \end{equation}
In Section \ref{sec.ddq} we show roughly that $\mu(H_{k} (D))$ decreases exponentially in $k$.

\subsection{Checking condition $\DD_q (u_n)$}\label{sec.ddq}
This argument is a minor adjustment of similar estimates in ~\cite{Dichotomy,FHN}. We state the following proposition.
\begin{proposition}\label{prop.ddq}
For every  $\ell,t,n\in \mathbb{N}$, there exists $\lambda_0\in(0,1)$, and $C>0$ such that
\[
\left|\mu\left(A_n^{(q)}\cap
  \mathscr{W}_{t,\ell}\left(A_n^{(q)}\right) \right)-\mu\left(A_n^{(q)}\right)
  \mu\left( \mathscr{W}_{0,\ell}\left(A_n^{(q)}\right)\right)\right|\leq C(n^{-2}+n^2\lambda^{t}_0).
\]
\end{proposition}
Condition $\DD_q (u_n)$ immediately follows from this. We can take $t_n=(\log n)^5$ so that
$n\gamma(q,n,t)\to 0$. 

The proof of Proposition~\ref{prop.ddq} is as follows. To check condition $\DD_q (u_n)$ we use decay of correlations. The main problem
in estimating the correlation function $\Theta_n(\Phi,\Psi)$ (recall equation \eqref{mixing}) is that the relevant indicator 
functions $\Phi=1_{A^{(q)}_n}$ and $\Psi=1_{\mathscr W_{0,\ell}\left(\A_n\right)}$ of the sets $A^{(q)}_n$ and
$\mathscr W_{0,\ell}\left(\A_n\right)$ are not Lipschitz continuous. 
Standard smoothing methods can be used to approximate $\Phi$, but $\Psi$ cannot be uniformly approximated
by a Lipschitz  function: the level set $\Psi=1$ has a geometry that becomes increasingly complex (i.e. with multiple connectivity) 
as $\ell$ increases. Fortunately, we can employ a further trick to approximate $\Psi$. This is done using a function that is constant on local stable manifolds.
This allows us to use a decay of correlations estimate using the $L^{\infty}$ norm. As part of this approximation we first estimate 
$\mu(H_k(D))$ with $D=A^{(q)}_n$. The geometry of the set $A^{(q)}_n$ will be important in calculating this estimate. 

\begin{lemma}\label{lemma:annulus1}
Consider the set $D=A^{(q)}_n$. Then there exists $C>0$ such that, for all $k$,
\begin{equation}\label{annulus}
 \mu(H_{k}(D))\le  C \lambda^{-k},
\end{equation}
where $\lambda^{-1}<1$ is the (uniform) contraction rate along the stable manifolds for the hyperbolic toral automorphism.
\end{lemma}
\begin{proof} We follow~\cite[Proposition 4.1]{Dichotomy}, and consider also the geometrical properties of $D$.
Since the local stable manifolds contract uniformly there exists
$C_1>0$ such that $\dist\,(T^n(x), T^n(y))\le C_1 | \lambda|^{-n}$ for all
$y\in W^s_1(x).$ This implies that $|T^k(W_1^s(x))|\le C_1\lambda^{-k}$. Therefore, for every  $x\in H_{k}(D)$,
 the leaf $T^k(W^s_1(x))$ lies in an tubular region of width $2 /|\lambda |^k$ around $\partial D$. To measure of the size of
this tube we note that $m(D_{\epsilon})\leq \epsilon C_{D}$, where $C=\epsilon c_q\ell_D$. (Again recall the definition
of the tubular region $D_{\epsilon}$ given in Section \ref{sec.anosov}).
The constant $c_q$ depends on
the number of connected components of $A^{(q)}_n$, (which is bounded), and $\ell_D$ is the maximum length of a connected component
of $\partial D$. This is also bounded, since $\partial D$ is formed of straight lines of bounded length. The lemma follows by taking 
$\epsilon=\lambda^{-k}$. 
\end{proof}

The next lemma also holds for $\{X_0>u_n\}$ in place of $A_n^{q} $, and the proof is the same as~\cite[Lemma 4.2]{Dichotomy}.
Again we give the main steps, indicating the role of Lemma \ref{lemma:annulus1}.
The constant $\tau_1$ in the next lemma comes from the exponential decay of correlations of Lipschitz observables on hyperbolic toral automorphisms.
\begin{lemma}\label{lemma:dun-prelim}
Suppose
 $\Phi:M\to \R$ is a Lipschitz map and $\Psi$ is the indicator function
 \[
\Psi:= \I_{\mathscr W_{0,\ell}\left(\A_n\right)}.
 \]
Then there exists $0<\tau_1<1$ such that for all $j\geq 0$
\begin{equation}
 \left|\int\Phi\,(\Psi\circ T^j)\, \text{d}\mu - \int\Phi\text{d}\mu \int\Psi\text{d}\mu \right|\le \C\,\left(\|\Phi\|_\infty \,
\lambda^{-\floor{j/2}}+\|\Phi\|_{\text{Lip}}\,\,\tau_{1}^{\floor{j/2}}\right).
\end{equation}
\end{lemma}
\begin{proof}
Following Lemma \cite[Lemma 4.2]{Dichotomy}, we take a version $\overline{\Psi}$ of $\Psi$ that is constant on local stable manifolds,
for example by taking a distinguished point $x^*$
on each local stable manifold $W^s_1 (x)$ and defining $\overline{\Psi}(y) =\Psi (x^*)$ for all $y\in W^s_1 (x)$.
We let $\Psi_j=\Psi\circ T^j$, and again denote $\overline{\Psi}_j$ as the relevant version of $\Psi_j$ (constant on local stable manifolds).
A simple application of the triangle inequality gives the following bound:
\begin{equation}
\Theta_j(\Phi,\Psi)\leq C\left(\|\Phi\|_{\infty}\mu \{\overline{\Psi}_{j/2}\neq\Psi_{j/2}\}+\|\Phi\|_{\mathrm{Lip}}\tau_{0}^{j/2}\right),
\end{equation}
(recall that $\Theta_j$ is defined in equation \eqref{mixing}).
To estimate $\mu \{\overline{\Psi}_{j/2}\neq\Psi_{j/2}\}$, we consider points $x_1,x_2$ on the same stable manifold, and such that
$x_1\in \mathscr W_{i,\ell}\left(\A_n\right)$, but $x_2\not\in\mathscr W_{i,\ell}\left(\A_n\right)$, (for $i\geq j/2$).
This set is contained in $\cup_{k=i}^{i+\ell-1}H_k(A^{q}_n)$. Hence 
\begin{equation*}
\mu \{\overline{\Psi}_{j/2}\neq\Psi_{j/2}\}\leq\sum_{k=j/2}^{\infty}H_k(A^{(q)}_n)\leq C\lambda^{-j/2}.
\end{equation*}
The conclusion of Lemma \ref{lemma:dun-prelim} follows.
\end{proof}

To continue with the proof of Proposition \ref{prop.ddq}, and hence verify condition $\DD_q (u_n)$, we approximate the characteristic function of the set $\A_n$ by a suitable Lipschitz function. The key estimate is to bound the  Lipschitz norm of the approximation.


Let $A_n=\A_n$ and $D_n:=\set{x\in \A_n:\; d_H\left(x, \overline{A_n^c}\right)\geq n^{-2}},$
 where $\bar A_n^c$ denotes the closure of the complement of the set $A_n$.
Define $\Phi_n:\X\to\R$ by
\begin{equation}
\label{eq:Lip-approximation}
\Phi_n(x)=\begin{cases}
  0&\text{if $x\notin A_n$}\\
  \frac{ d_H(x,A_n^c)}{d_H(x,A_n^c)+
  d_H(x,D_n)}&
  \text{if $x\in A_n\setminus D_n$}\\
  1& \text{if $x\in D_n$}
\end{cases}.
\end{equation}
Note that $\Phi_n$ is Lipschitz continuous with Lipschitz constant given by $n^2$. 
Moreover  $\|\Phi_n-\I_{A_n}\|_{L^1(m)}\leq C/n^2$ for some constant $C$.
%
It follows that
\begin{align}
\Big|\int \I_{\A_n}\,\left(\Psi_{\floor{j/2}}\circ T^{j -\floor{j/2}}\right) &\,d\mu - \mu (\A_n)\int \Psi d\mu \Big|\nonumber\\
&\le \abs{\int \left(\I_{\A_n} - \Phi_n\right)\Psi_{\floor{j/2}}\,d\mu }+\C \left(\|\Phi_n \|_\infty \,\,j^2\,\,\lambda^{\floor{j/4}}+\|\Phi_n \|_{\text{Lip}}\,\,\tau_{1}^{\floor{j/2}}\right)\nonumber\\
&\quad+ \abs{\int \left(\I_{\A_n} - \Phi_n \right)\,d\mu \int \Psi_{\floor{j/2}}\,d\mu },
\end{align}
for some generic constant $\C$. Thus
\[
 \abs{\mu \left(\A_n\cap \mathscr W_{j,\ell}(\A_n)\right) - \mu (\A_n)\,\mu \left(\mathscr W_{0,\ell}(\A_n)\right)}\le \gamma(n,j)
\] where
\[
 \gamma(n,j) = \C\,\left(n^{-2}+ n^{2} \,\lambda_1^{\floor{j/2}}\right),
\] and
\[
\lambda_1 = \max\,\set{\tau_1 , \lambda^{-1}}.\]
Thus if, for instance, we choose  $j=t_n=(\log n)^{5}$, then $n\gamma(n, t_n)\to 0$ as $n\to\infty.$ This completes the proof.

\subsection{Checking condition $\DD'_q(u_n)$}\label{sec.DDprime}

We make  the following decomposition:
\begin{multline*}
n\sum_{j=q+1}^{\lfloor n/k_n\rfloor} \mu (A^{(q)}_n\cap T^{-j}(A^{(q)}_n))=n\sum_{j=q+1}^{R_n} \mu (A^{(q)}_n\cap T^{-j} (A^{(q)}_n))+n\sum_{j=R_n+1}^{(\log n)^5} \mu (A^{(q)}_n\cap T^{-j}(A^{(q)}_n))\\ +n\sum_{(\log n)^5+1}^{\lfloor n/k_n\rfloor} \mu (A^{(q)}_n\cap T^{-j}(A^{(q)}_n)),
\end{multline*}
where the sequence $R_n\to\infty$ (as $n\to\infty$) will be chosen later. Recall that for $q=0$,
$A^{(q)}_n=U_n$. By exponential decay of correlations and a suitable Lipschitz approximation the last sum tends to 0 as $n\to\infty$,
so it suffices to estimate the  two sums  where $1\le j\le (\log n)^5$.

\noindent {\bf Case $L$ transverse to stable and unstable directions.}

Fix $y$ and define $u_n(y)$ by the requirement $n \mu \{ x: \phi (x) \ge u_n (y)  \}=y$. Henceforth we will drop the dependence on $y$ and write 
simply $u_n$ for convenience. We define
$U_n:= \{ x: \phi (x) \ge u_n \}$. Geometrically $U_n$ resembles a parallel strip of width  $\frac{2}{n}$. 
 
We will verify the short return condition with $q=0$. Consider the set $T^{-j} U_n\cap U_n=\{ x: T^j (x) \in U_n, x\in U_n \}$.   $T^{j} U_n$
   is a union of parallelogram-like strips  corresponding to each winding around the torus and such strip has  width $O(\frac{\lfloor\lambda^{-j}\rfloor}{n})$ and 
 length $O(1)$, the precise constants depending on the angle between $T^j L$ and $L$ as $T^j L$ aligns to the unstable
 direction. There are
 approximately $\lfloor\lambda^j\rfloor$ such parallelogram strips. Each strip intersects $U_n$ in an area of measure $O(\lfloor \lambda^{-j} \rfloor n^{-2})$ by transversality. See Figure \ref{fig.anosov1}
 
 Hence 
 \[
n \sum_{j=1}^{(\log n)^5} \mu ( T^{-j} U_n\cap U_n )=O\left(\frac{(\log n)^5}{n}\right).
 \]
  
 Thus the extremal index $\theta=1$.

\begin{figure}[h!]\label{fig.transverse}
\centering
\begin{minipage}{0.49\textwidth}
\centering
\begin{tikzpicture}[scale=4]
\begin{scope}
\draw[thick] (0,0) rectangle (1,1);
\clip (0,0) rectangle (1,1);
\draw[thick,rotate=pi/2] (0,0)--(1,1);
\draw[thick,rotate=pi/2] (1,0)--(0,1);
\draw[thick,black,->,shorten >=1cm,rotate=pi/2] (0,0)--(1,1);
\draw[thick,black,->,shorten >=1.25cm,rotate=pi/2] (0,0)--(1,1);
\draw[thick,black,<-,shorten <=1cm,rotate=pi/2] (0,0)--(1,1);
\draw[thick,black,<-,shorten <=1.25cm,rotate=pi/2] (0,0)--(1,1);
\draw[thick,black,<-,shorten <=3cm,rotate=pi/2] (1,0)--(0,1);
\draw[thick,black,<-,shorten <=3.25cm,rotate=pi/2] (1,0)--(0,1);
\draw[thick,black,->,shorten >=3cm,rotate=pi/2] (1,0)--(0,1);
\draw[thick,black,->,shorten >=3.25cm,rotate=pi/2] (1,0)--(0,1);
\end{scope}
\node at (0.75,0.28) [right] {$v^{-}$};
\node at (0.9,0.9) [below] {$v^{+}$};
\draw[thick] (0,0.1)--(0.6,1);
\draw[thick,->,dotted] (0,0.8)--(0.53333,0);
\node at (0,0.8) [left] {$v$};
\node at (0.75,1) [above] {$x\cdot v=c$};
\draw[fill = gray,opacity=0.2] (0,0)--(0,0.2)--(0.53333,1)--(0.6667,1)--(0,0);
\draw[thick,|-|] (0,0)--(0,0.2);
\node at (0,0.1) [left] {$U_n$};
\end{tikzpicture}
\caption*{(a)}
\end{minipage}
\begin{minipage}{0.49\textwidth}
\raggedright
\begin{tikzpicture}[scale=4]
\begin{scope}
\draw[thick] (0,0) rectangle (1,1);
\clip (0,0) rectangle (1,1);
\draw[thick,opacity=0.1,rotate=pi/2] (0,0)--(1,1);
\draw[thick,opacity=0.1,rotate=pi/2] (1,0)--(0,1);
\draw[thick,black,->,shorten >=1cm,opacity=0.1,rotate=pi/2] (0,0)--(1,1);
\draw[thick,black,->,shorten >=1.25cm,opacity=0.1,rotate=pi/2] (0,0)--(1,1);
\draw[thick,black,<-,shorten <=1cm,opacity=0.1,rotate=pi/2] (0,0)--(1,1);
\draw[thick,black,<-,shorten <=1.25cm,opacity=0.1,rotate=pi/2] (0,0)--(1,1);
\draw[thick,black,<-,shorten <=3cm,opacity=0.1,rotate=pi/2] (1,0)--(0,1);
\draw[thick,black,<-,shorten <=3.25cm,opacity=0.1,rotate=pi/2] (1,0)--(0,1);
\draw[thick,black,->,shorten >=3cm,opacity=0.1,rotate=pi/2] (1,0)--(0,1);
\draw[thick,black,->,shorten >=3.25cm,opacity=0.1,rotate=pi/2] (1,0)--(0,1);

\draw[thick,opacity=0.3,rotate=pi/2] (0,0.1)--(0.6,1);
\draw[fill = gray,opacity=0.3,rotate=pi/2] (0,0)--(0,0.2)--(0.53333,1)--(0.6667,1)--(0,0);
\draw[opacity=0.3,rotate=pi/2] (0,0)--(0,0.2)--(0.53333,1)--(0.6667,1)--(0,0);
\draw[fill = gray, opacity=0.3,rotate=-pi/2] (0,0.3)--(0.7,1)--(1,1.25)--(0,0.25)--(0,0.3);
\draw[rotate=-pi/2] (0,0.3)--(0.7,1)--(1,1.25)--(0,0.25)--(0,0.3);
\draw[fill = gray, opacity=0.3,rotate=-pi/2] (0,0.35)--(0.65,1)--(0.95,1.25)--(0,0.3)--(0,0.35);
\draw[rotate=-pi/2] (0,0.35)--(0.65,1)--(0.95,1.25)--(0,0.3)--(0,0.35);
\end{scope}
\draw[|-|,thick] (0.53333,1)--(0.6667,1);
\node at (0.6,1) [above] {$O(\frac{1}{n})$};
\draw[|-|,thick] (0,0.3)--(0,0.35);
\node at (0,0.325) [left] {$O(\frac{\lambda^{-j}}{n})$};
\draw [<-,very thick] (0.45,0.75)--(0.65,0.55);
\node at (0.65,0.55) [below] {$T^{j} U_n\cap U_n$};
\end{tikzpicture}
\caption*{(b)}
\end{minipage}
\caption{\label{fig.anosov1} (a) The set $U_n$ and line $L$ for $L$ not aligned with $v^{-}$ or $v^{+}$. (b) Iterations $T^j U_n$ and their intersections with $U_n$.}
\end{figure}

\noindent {\bf Case $L$ aligned with unstable direction.}

 We lift $L$ to $\hat{L}$  on  a fundamental domain of the cover $\R^2$ of $\T^2$ and write $\hat{L}=\hat{p}_1+ tv^+$, $t\in [0,l(L)]$, $\hat{p}_1\in \R^2$. We write the 
endpoint of $\hat{L}$ as $\hat{p}_2$, i.e. $\hat{p}_2=\hat{p}_1+l(L)v^+$. The points $\hat{p}_1$  and $\hat{p}_2$ project to  the corresponding points 
written $p_1=\pi \hat{p_1}$ and $p_2=\pi \hat{p_2} $.

There are 2 main cases, with some subcases. 

Case (a): First assume that the line $\hat{p}_1+t v^+$, $-\infty <  t<\infty$  contains no point with rational coordinates. This holds for 
a measure one set of $\hat{p}_1$ as the set of points in the plane with rational coordinates is countable. In this case $T^nL$, $n\ge 1$,  has no intersections with $L$. To see this suppose $p\in L$ and there exists an $n$ such that
$T^n p=q\in L$. If we take a line segment $\tilde{L}$  in direction $v^+$ of length $2l(L)$  centered at $p$ we see
by expansion that  $\tilde{L}\subset T^n \tilde{L}$  (since $d(p,q)\le l(L)$) and hence $T^n$ restricted to $\tilde{L}$ has a fixed point $\tilde{p}$ in $\tilde{L}$. However, this implies the lift  $\hat{p}_1+t v^+$, $-\infty <t<\infty$ contains a  point with rational coordinates, which is a contradiction.

Since $p_1$ is not periodic by assumption and $\hat{p}_1$ is not in the direction of $v^+$ (otherwise the point $(0,0)$
would be contained in $\hat{p}_1+t v^+$, $-\infty <t < \infty$) the iterates $T^j U_n$ are disjoint for large $n$ for small $j$  i.e.
there exists $R_n \to \infty$ such that $\mu (T^{-j} U_n\cap U_n)=0$ for $j<R_n$. Corollary 2.2 of the recent preprint~\cite{Fan_Yang} implies in this case that the extremal index is one.
We include an alternate  proof for completeness.

For large $n$ the set $T^j U_n$ comprises
$\lfloor \lambda^j \rfloor$ parallel rectangles (aligned with the unstable direction) of width $O(\frac{\lfloor \lambda^{-j}\rfloor}{n})$.
Identifying $\mathbb{T}^2$ with the unit square the set $T^j L\cap ([0,1]\times \{0\})$ consists of $m(j)\sim[\lambda^j]$ points $x_i^j$, $j=1,\ldots, m(j)$.
If for small iterates $T^i L$ there is no intersection with  $([0,1]\times \{0\})$ we extend $T^i L$ in a straight line so that all $x_i^j$, $j=1,\ldots, m(j)$
are defined.  Let $\gamma^{-1}$ denote the slope of $v^+$.
The  set $\{ x^j_i\}_{i=1,\ldots, m(j)}$ is generated by the relation $x_1^j+k\gamma $(mod $1$) for $k=1,\ldots, m(j)$.

We now estimate $\mu (T^{-j} U_n\cap U_n)$. The set $T^j U_n$ has approximately $[\lambda^j]$ windings around the torus
and we now estimate the fraction of these that intersect $U_n$.

Note that $\gamma$ is a quadratic irrational. This implies that $\gamma$ has low discrepancy in the sense that there exists a constant $C>0$
such that 
\[
\sup_{0\le a  <b  \le 1} \{ \# \{x^j_i \in (a,b)\}/\lfloor \lambda^j\rfloor  - (b - a) \} \le C \frac{\log \lfloor \lambda^j\rfloor }{\lfloor \lambda^j\rfloor },
\]
see \cite{N92}.
Hence for $j> R_n$
\[
n\sum_{j=R_n}^{(\log n)^5} \mu (T^{-j} U_n\cap U_n)=O\left(\frac{1}{n}+\frac{\log [\lambda^{R_n}]}{\lambda^{R_n}}\right)=o(1).
\]
This implies that a standard EVL holds with  $\theta=1$. See Figure \ref{fig.anosov2}.

Case (b): Assume that $\hat{p}_1+t v^+$, $-\infty <t<\infty$ contains a point with rational coordinates, note that it will contain at most one as the slope of
$v^+$ is irrational. Such a point projects to a point $p_{per}$ periodic under $T$ with period $q$ say.

Case (b1):  Assume now that $L $ itself contains  $p_{per}$, a periodic point of period $q$.

There will be only one periodic point in $L$ as the slope of $v^+$ is irrational.
Without loss of generality we take $q=1$ by considering $T^q$.
It is easy to see that $\theta=\lim_{n\to\infty}\frac{\mu(A^q_n)}{\mu(U_n)}=1-\frac{1}{\lambda^q}$. 
The same discrepancy argument as in the case of no periodic orbits shows that there exists an $R_n\to \infty$ such that $T^{-j} A_n^{(q)} \cap A_n^{(q)}=\emptyset$ for $j<R_n$ and 

\[
\sum_{j=R_n}^{(\log n)^5} \mu (T^{-j} A_n^{(q)} \cap A^{(q)}_n)= o\left(\frac{1}{n}\right)
\]
Hence $\theta=1-\frac{1}{\lambda^q}$. See Figure \ref{fig.anosov3}.

Case (b2):   $L $ does not contain a  periodic point.

We first  consider the simplest case where the origin is the fixed point and $\hat{p}_1$ parallel to $v^+$ so that $\hat{p}_1+t v^+$, $-\infty <t<\infty$ contains the  fixed point $(0,0)$
but $L$ does not contain $(0,0)$. The line $\hat{p}_1+ t v^+$, $0\le t<\infty$ has a natural ordering by distance from the origin $(0,0)$. If $\lambda \hat{p}_1 > \hat{p}_2$
then it is easy to see all iterates of $T^n L$ on the torus are disjoint and the arguments given in case (a) apply giving $\theta =1$. 

Suppose now $\lambda \hat{p}_1 < \hat{p}_2$. We take $q=1$ and calculate 
$$\theta =\mu (A_{n}^{(1)})/\mu (U_n)=|\hat{p}_2-\frac{1}{\lambda}\hat{p}_2|/|\hat{p}_2-\hat{p}_1|=(1-\frac{1}{\lambda})\frac{|\hat{p}_2|}{|\hat{p}_2-\hat{p}_1|},$$ as the 
stable manifolds are sent strictly into the region of intersection $U_n\cap TU_n$. (See Figure \ref{fig.anosov4}). The condition $\lambda \hat{p}_1 <\hat{p}_2$ implies
$1< \frac{|\hat{p}_2|}{|\hat{p}_2-\hat{p}_1|} < (1-\frac{1}{\lambda})^{-1}$. By varying $\hat{p}_1$ and $\hat{p}_2$ we may obtain all values in this range.  Hence 
$(1-\frac{1}{\lambda})\le \theta \le 1$.

In the general case of a periodic point $p_{per}$ of period $q$ contained in $\pi (\hat{p}_1+t v^+)$, $-\infty <t<\infty$ we consider 
$T^q$ and the analysis proceeds in the same way by considering the expansion on the  line segment $[\hat{p}_1-\hat{p}_{per},\hat{p}_2-\hat{p}_{per}]$.  We  infer that for general $q\geq 1$, 
$$(1-\frac{1}{\lambda^q})\le \theta \le 1$$
with all values of $\theta$ in this range being realizable. 
The verification 
of condition $\DD'_q(u_n)$ is similar to  case (b1).

\begin{figure}[h!]
\centering
\begin{tikzpicture}[scale=4]
\begin{scope}
\draw[thick] (0,0) rectangle (1,1);
\clip (0,0) rectangle (1,1);
\draw[thick,rotate=pi/2] (0,0)--(1,1);
\draw[thick,rotate=pi/2] (1,0)--(0,1);
\draw[thick,black,->,shorten >=1cm,rotate=pi/2] (0,0)--(1,1);
\draw[thick,black,->,shorten >=1.25cm,rotate=pi/2] (0,0)--(1,1);
\draw[thick,black,<-,shorten <=1cm,rotate=pi/2] (0,0)--(1,1);
\draw[thick,black,<-,shorten <=1.25cm,rotate=pi/2] (0,0)--(1,1);
\draw[thick,black,<-,shorten <=3cm,rotate=pi/2] (1,0)--(0,1);
\draw[thick,black,<-,shorten <=3.25cm,rotate=pi/2] (1,0)--(0,1);
\draw[thick,black,->,shorten >=3cm,rotate=pi/2] (1,0)--(0,1);
\draw[thick,black,->,shorten >=3.25cm,rotate=pi/2] (1,0)--(0,1);
\node at (0.75,0.28) [right] {$v^{-}$};
\node at (0.9,0.9) [below] {$v^{+}$};
\draw[thick,rotate=pi/2] (0,0.1)--(0.9,1);
\draw[thick,<-,dotted,rotate=pi/2] (0.9,0)--(0,0.9);
\node at (0.8,0.1) [left] {$v$};
\draw[fill = gray,opacity=0.2,rotate=pi/2] (0,0)--(1,1)--(0.8,1)--(0,0.2)--(0,0);
\end{scope}
\node at (0.9,1) [above] {$x\cdot v=c$};
\draw[thick,|-|] (0,0.2)--(0,0);
\node at (0,0.1) [left] {$U_n$};
\end{tikzpicture}
\caption{\label{fig.anosov2} The set $U_n$ and line $L$ for $L$ aligned with $v^{-}$.} 
\end{figure}
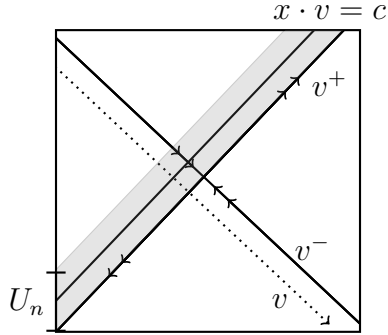


\begin{figure}[h!]
\centering
\begin{tikzpicture}[scale=4]
\begin{scope}
\draw[dashed] (0,-0.6)--(0,0.6);
\draw (-0.75,0)--(0.75,0);
\draw (-0.5,-0.25) rectangle (0.5,0.25);
\draw[fill=gray,opacity=0.3] (-0.65,-0.15) rectangle (0.65,0.15);
\draw[pattern = north west lines] (0.4,-0.25) rectangle (0.5,0.25);
\draw[pattern = north west lines] (-0.5,-0.25) rectangle (-0.4,0.25);
\draw[|-|] (-0.5,-0.3) -- (-0.4,-0.3);
\node at (-0.45,-0.3) [below] {$A_n^{(1)}$};
\draw[|-|] (0.4,-0.3) -- (0.5,-0.3);
\node at (0.45,-0.3) [below] {$A_n^{(1)}$};
\draw[fill = black] (0,0) circle (0.35pt);
\node at (0.75,0) [right] {$v = v^{+}$};
\node at (0,0.6) [above] {$v^{-}$};
\draw[decorate,decoration={brace,amplitude=10pt}] (-0.5,0.25)--(0.5,0.25);
\node at (0,0.3) [above] {$U_n$};
\draw[decorate,decoration={brace,amplitude=10pt}] (-0.65,-0.15)--(-0.65,0.15);
\node at (-0.7,0) [left] {$T(U_n)$};
\draw[<->] (-0.5,-0.5)--(0.5,-0.5);
\node at (0,-0.5) [above] {$1$};
\draw[<->] (-0.4,-0.55)--(0.4,-0.55);
\node at (0,-0.55) [below] {$\frac{1}{\lambda}$};
\end{scope}
\end{tikzpicture}
\caption{\label{fig.anosov3} Sketch of argument (b1) for $v$ aligned with the unstable direction and $L$ contains a periodic orbit showing intersections of $A_n^{(1)}$ (shown in patterned lines) and $T(U_n)$ (shown in gray). Estimates of the ratio of $A_n^{(1)}$ to $U_n$ (shown in white) give the value of the extremal index.}
\end{figure}

\begin{figure}[h!]
\centering
\begin{tikzpicture}[scale=4]
\begin{scope}
\draw[dashed] (0,-0.6)--(0,0.6);
\draw (-0.6,0)--(0.75,0);
\draw (-0.4,-0.3) rectangle (0.3,0.3);
\draw[fill = gray, opacity = 0.3] (-0.3,-0.15) rectangle (0.65,0.15);
\draw[pattern = north west lines] (0.2,-0.3) rectangle (0.3,0.3);
\draw[fill = black] (-0.4,0) circle (0.35pt);
\draw[fill = black] (0.2,0) circle (0.35pt);
\draw[fill = black] (0.3,0) circle (0.35pt);
\draw[decorate,decoration={brace,amplitude=10pt}] (-0.4,0.3)--(0.3,0.3);
\draw[decorate,decoration={brace,amplitude=10pt}] (0.65,0.15)--(0.65,-0.15);
\node at (-0.05,0.35) [above] {$U_n$};
\node at (0.7,0) [right] {$T(U_n)$};
\node at (-0.6,0) [left] {$v = v^{+}$};
\node at (0,0.6) [above] {$v^{-}$};
\draw [|-|] (0.2,-0.35)--(0.3,-0.35);
\node at (0.25,-0.35) [below] {$A_n^{(1)}$};
\draw[<-] (-0.4,0)--(-0.45,0.2);
\node at (-0.45,0.2) [above] {$\hat{p}_1$};
\draw[<-] (0.2,0) -- (0,0.15);
\node at (0,0.15) [above] {$T^{-1}(\hat{p}_2)$};
\draw[<-] (0.3,0) -- (0.4,0.15);
\node at (0.4,0.15) [above] {$\hat{p}_2$};
\end{scope}
\end{tikzpicture}
\caption{\label{fig.anosov4} Sketch of argument (b2) for $v$ aligned with the unstable direction and $L$ does \textit{not} contain a periodic orbit showing intersections of $A_n^{(1)}$ (shown in patterned lines) and $T(U_n)$ (shown in gray). Estimates of the ratio of $A_n^{(1)}$ to $U_n$ (shown in white) give the value of the extremal index.}
\end{figure}
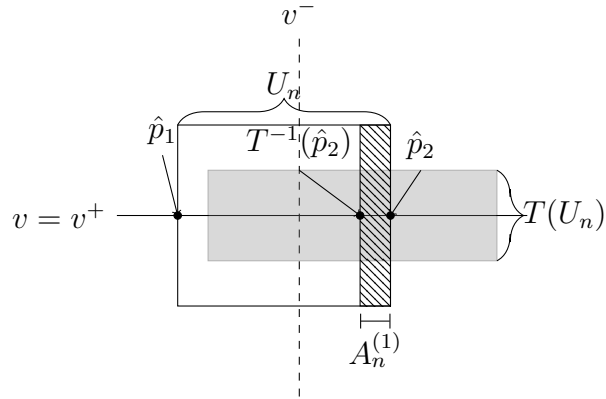
																																																													
																																																													            \noindent {\bf Case $L$ is aligned with the stable direction.}

Suppose now that $L$ aligns with the stable direction $v^{-}$. See Figure \ref{fig.anosov5}. The analysis is similar to the 
case where $L$ is aligned with the unstable direction, and again we consider the lift $\hat{L}=\hat{p}_1+ tv^-$, $t\in [0,l(L)]$, 
with $\hat{p}_1\in \R^2$, and $\hat{p}_2$ denoting the other endpoint of $\hat{L}$, i.e. $\hat{p}_2=\hat{p}_1+l(L)v^-$. We will make use of the time-reversibility of the system in Case (a) below. 

We have the following cases.

Case (a): First assume that the line $\hat{p}_1+t v^-$, $-\infty <  t<\infty$  contains no point with rational coordinates. Let $S=T^{-1}$. Then $L$
is aligned with the unstable direction for $S$.
As in the case where $L$ aligned with the unstable direction for $T$, it follows again that 
$S^n (L)$ has no intersections with $L$, for all $n\geq 1$. Hence $T^n(L)$ has no intersections with $L$ for all $n\geq 1$.

Thus all the iterates $T^j U_n$ are disjoint for small 
$j$, i.e. there exists $R_n \to \infty$ such that $\mu (T^{-j} U_n\cap U_n)=0$ for $j<R_n$. Note that the definition of $U_n$ is the same for $T$ and $S$ 
and that $\mu (T^{-j} U_n \cap U_n)= \mu ( U_n \cap T^j U_n)=
\mu (U_n\cap S^{-j} U_n)$
by measure-preservation. The argument of Case (a) when $L$ is aligned with the unstable direction shows that 
\[
n \sum_{j=R_n}^{(\log n)^5} \mu (S^{-j} U_n \cap U_n)= o(1),
\]
and hence 
\[
n \sum_{j=R_n}^{(\log n)^5} \mu (T^{-j} U_n \cap U_n)= o(1),
\]
Thus  $\theta=1$.

Case (b): Assume that $\hat{p}_1+t v^-$, $-\infty <t<\infty$ contains a point $\hat{p}_{per}$ with rational coordinates.  There will be only
one such point as the slope of $v^{-}$ is irrational.  The  point  $\hat{p}_{per}$ projects to a point $p_{per}$ periodic under $T$ with period $q$ say. We cannot
use time-reversibility in this case as the set $A_n^q$ depends upon the consideration of $T$ or $T^{-1}$ as the transformation. 

Case (b1):  Assume now that $L$ contains the periodic  point $p_{per}$ of period $q$.

Without loss of generality we (again) take $q=1$ by considering $T^q$.
We have $\theta=\lim_{n\to\infty}\frac{\mu(A^{(q)}_n)}{\mu(U_n)}=1-\frac{1}{\lambda^q}$. Geometrically $A_n^{(q)}$ consists of two
strips within $U_n$. Both of these have length $1/n$, (i.e. the same as $U_n$), but their width relative to $U_n$ is 
$\frac{1}{2}\left(1-1/\lambda^q\right)$. See Figure \ref{fig.anosov6}.
The same argument as in the case of no periodic orbits shows that there exists an $R_n\to \infty$ such that $T^{-j} A_n^{(q)} \cap A_n^{(q)}=\emptyset$ for $j<R_n$.

We have uniform expansion of $A_n^q$ in the unstable direction and the discrepancy argument of Case (b1)  of the previous section (alignment with the 
unstable direction) shows that 
\[
\sum_{j=R_n}^{(\log n)^5} \mu (T^{-j} A_n^{(q)} \cap A^{(q)}_n)= o\left(\frac{1}{n}\right).
\]
We therefore have $\theta=1-\frac{1}{\lambda^q}$.

Case (b2):  $L$ does not contain a  periodic orbit.

Again, we illustrate by considering the simplest case of  $\hat{p}_1$ parallel to $v^-$ so that $\hat{p}_1+t v^-$, $-\infty <t<\infty$ contains the fixed point $(0,0)$
but $L$ does not contain $(0,0)$. For the lifted line $\hat{p}_1+ t v^-$, $0\le t<\infty$, we use the natural ordering by distance 
from the origin $(0,0)$. If $\hat{p}_1 > \lambda^{-1}\hat{p}_2$,
then all iterates of $T^n L$ on the torus are disjoint and the arguments given in case (a) apply giving $\theta =1$. 

Suppose now $\hat{p}_1 < \lambda^{-1}\hat{p}_2$. We take $q=1$ and calculate 
$$\theta =\frac{\mu (A_n^{(1)})}{\mu (U_n)}
=\left(1-\frac{1}{\lambda}\cdot \frac{|\hat{p}_2-\lambda^{-1} \hat{p}_1|}{\hat{p}_2-\hat{p}_1}\right).$$ 
See Figure \ref{fig.anosov7}.
The general case where $\hat{p}_1$ is not parallel $v^-$ proceeds the 
same way by considering the  expansion of $T$ orthogonal to the  segment $[\hat{p}_1-\hat{p}_{per},\hat{p}_2-\hat{p}_{per}]$.  We  infer that for general $q\geq 1$, 
$$(1-\frac{1}{\lambda^q})\le \theta \le 1$$
with all values of $\theta$ in this range being realizable. 
 The verification of condition $\DD'_q(u_n)$ is similar to case (b1).

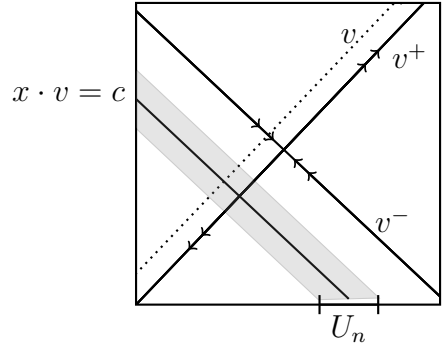
\begin{figure}[h!]
\centering
\begin{tikzpicture}[scale=4]
\begin{scope}
\draw[thick] (0,0) rectangle (1,1);
\clip (0,0) rectangle (1,1);
\draw[thick,rotate=pi/2] (0,0)--(1,1);
\draw[thick,rotate=pi/2] (1,0)--(0,1);
\draw[thick,black,->,shorten >=1cm,rotate=pi/2] (0,0)--(1,1);
\draw[thick,black,->,shorten >=1.25cm,rotate=pi/2] (0,0)--(1,1);
\draw[thick,black,<-,shorten <=1cm,rotate=pi/2] (0,0)--(1,1);
\draw[thick,black,<-,shorten <=1.25cm,rotate=pi/2] (0,0)--(1,1);
\draw[thick,black,<-,shorten <=3cm,rotate=pi/2] (1,0)--(0,1);
\draw[thick,black,<-,shorten <=3.25cm,rotate=pi/2] (1,0)--(0,1);
\draw[thick,black,->,shorten >=3cm,rotate=pi/2] (1,0)--(0,1);
\draw[thick,black,->,shorten >=3.25cm,rotate=pi/2] (1,0)--(0,1);
\node at (0.75,0.28) [right] {$v^{-}$};
\node at (0.9,0.9) [below] {$v^{+}$};
\draw[thick,->,dotted,rotate=pi/2] (0,0.1)--(0.9,1);
\draw[thick,rotate=pi/2] (0.7,0)--(0,0.7);
\node at (0.7,0.83) [above] {$v$};
\draw[fill = gray,opacity=0.2,rotate=pi/2] (0,0.6)--(0,0.8)--(0.8,0)--(0.6,0)--(0,0.6);
\end{scope}
\node at (0,0.7) [left] {$x\cdot v=c$};
\draw[thick,|-|] (0.6,0)--(0.8,0);
\node at (0.7,0) [below] {$U_n$};
\end{tikzpicture}
\caption{\label{fig.anosov5} The set $U_n$ and line $L$ for $L$ aligned with $v^{+}$.} 
\end{figure}

\begin{figure}[h!]
\centering
\begin{tikzpicture}[scale=4]
\begin{scope}
\draw[dashed] (0,-0.6)--(0,0.6);
\draw (-0.6,0)--(0.75,0);
\draw (-0.5,-0.25) rectangle (0.5,0.25);
\draw[fill = gray, opacity = 0.3] (-0.25,-0.5) rectangle (0.25,0.5);
\draw[pattern=north west lines]  (-0.5,-0.25) rectangle (0.5,-0.15);
\draw[pattern=north west lines] (-0.5,0.15) rectangle (0.5,0.25);
\draw[fill = black] (0,0) circle (0.35pt);
\node at (0.75,0) [right] {$v = v^{-}$};
\node at (0,0.6) [above] {$v^{+}$};
\draw[<->] (-0.55,0.25)--(-0.55,-0.25);
\node at (-0.55,-0.25) [below] {$1/n$};
\draw[<->] (-0.6,0.15)--(-0.6,-0.15);
\node at (-0.6,0) [left] {$\frac{1}{\lambda n}$};
\draw[|-|] (0.55,0.25) -- (0.55,0.15);
\node at (0.55,0.2) [right] {$A_n^{(1)}$};
\draw[|-|] (0.55,-0.25) -- (0.55,-0.15);
\node at (0.55,-0.2) [right] {$A_n^{(1)}$};
\draw[decorate,decoration={brace,amplitude=8pt},rotate=180] (-0.25,0.5)--(0.25,0.5);
\node at (0,-0.55) [below] {$T(U_n)$};
\draw[decorate,decoration={brace,amplitude=10pt}] (-0.5,0.25)--(0.5,0.25);
\node at (0,0.3) [above] {$U_n$};
\end{scope}
\end{tikzpicture}
\caption{\label{fig.anosov6} Sketch of argument (b1) for $v$ aligned with the stable direction and $L$ contains a periodic orbit showing intersections of $A_n^{(1)}$ (shown in patterned lines) and $T(U_n)$ (shown in gray). Estimates of the ratio of $A_n^{(1)}$ to $U_n$ (shown in white) give the value of the extremal index.}
\end{figure}

\begin{figure}[h!]
\centering
\begin{tikzpicture}[scale=4]
\begin{scope}
\draw[dashed] (0,-0.6)--(0,0.6);
\draw (-0.75,0)--(0.75,0);
\draw[fill = gray, opacity = 0.3] (-0.4,-0.3) rectangle (0.4,0.3);
\draw[fill=white] (0,-0.15) rectangle (0.6,0.15);
\draw[pattern = north west lines] (0,-0.15) rectangle (0.6,0.15);
\draw[<->] (0.65,0.15)--(0.65,-0.15);
\node at (0.65,0) [right] {$\frac{1}{\lambda n}$};
\node at (-0.75,0) [left] {$v = v^{-}$};
\node at (0,0.6) [above] {$v^{+}$};
\draw[fill = white] (0.3,-0.1) rectangle (0.6,0.1);
\draw[<-] (0.6,0) -- (0.6,0.45);
\node at (0.6,0.45) [above] {$\hat{p}_2$};
\draw[fill=black] (0,0) circle (0.35pt);
\draw[fill=black] (0.3,0) circle (0.35pt);
\draw[fill=black] (0.6,0) circle (0.35pt);
\draw[<-] (0.3,0) -- (0.4,-0.45);
\node at (0.4,-0.45) [below] {$T^{-1}(\hat{p}_2)=\lambda\hat{p}_2$};
\draw[<-] (0,0)--(-0.2,0.45);
\node at (-0.2,0.45) [above] {$\hat{p}_1$};
\draw[decorate,decoration={brace,amplitude=10pt}] (0,0.15)--(0.6,0.15);
\node at (0.3,0.2) [above] {$U_n$};
\draw[decorate,decoration={brace,amplitude=10pt}] (-0.4,0.3)--(0.4,0.3);
\node at (0,0.35) [above] {$T(U_n)$};
\end{scope}
\end{tikzpicture}
\caption{\label{fig.anosov7} Sketch of argument (b2) for $v$ aligned with the stable direction and $L$ does \textit{not} contain a periodic orbit. Showing intersections of $A_n^{(1)}$ (shown in patterned lines) and $T(U_n)$ (shown in gray). Estimates of the ratio of $A_n^{(1)}$ to $U_n$ (shown in white) give the value of the extremal index.}
\end{figure}
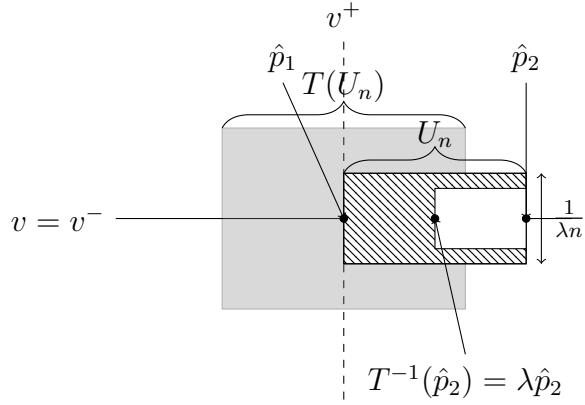

\subsection{Proof of Theorem~\ref{thm:billiards}.}\label{sec:billiards}

We will show that conditions $\DD_q(u_n)$ and $\DD'_q (u_n)$ hold with $q=0$ so that the extremal index $\theta=1$. We shall drop the subscript $q$ in this section. The proof of $\DD(u_n)$ follows the same strategy as in the Anosov case, the differences necessary in the planar dispersing billiard setting are addressed in~\cite[Theorem 2.1]{GHN}. To simplify the exposition we will consider the case $L=\{x: r=r_0\}$. The proof in the general case
of a $C^1$ curve is  similar.

\subsection{Checking condition $\DD^{'}(u_n)$}
 Before checking $\DD^{'}(u_n)$, we note that we need only to consider the sum up to time $(\log n)^{1+\delta}$, for $\delta>0$ since by the exponential decay of correlations of Lemma 3.3 (with $(X_0>u_n)$ equal to $A_n^q$ in this case), the remaining sum $$n\sum_{j= (\log n)^{1+\delta}}^{\lfloor k_n/n\rfloor} \mu (U_n \cap T^{-j} U_n)\to 0.$$
(Note here, we work with $A^{(0)}_n\equiv U_n$).

 The set  $\{r=r_0\}$ corresponds to a line (call it $L$) which is transverse to the discontinuity set  $S^+$ for $T$ and the discontinuity $S^{-}$  for $T^{-1}$. Let $U_n$ be the rectangle centered at $L$ with length $\pi$ and of width roughly $\frac{\tau}{\pi n}$ corresponding to the set $\{\phi > u_n\}$ so that $\mu(U_n) = \frac{\tau}{n}$.

\par\noindent\textbf{Short Returns.}

Let $S_n=\cup_{j=0}^{n-1}T^{-j} S^+$. The number of smooth connected components of $S_{n}$ is bounded above by $\kappa^n$ for some $\kappa>0$. Let $C=\frac{1}{4\log \kappa}$
and then the number of smooth connected components in $S_{[C\log n]}$ is bounded above by $n^{1/4}$.  Let $p_i=(r_0,\vartheta_i) \in L$ be the intersection points
$S_{[C\log n]}\cap L$, ordered from lowest $\vartheta$ value to highest and let $\alpha_i=\vartheta_{i+1}-\vartheta_i$. Let $B_1=\{ \alpha_i: \alpha_i < n^{-1/2} \}$.
We estimate $\sum_{\alpha_i \in B_1} \alpha_i
\le n^{1/4}n^{-1/2}=n^{-1/4}$. For each $\alpha_i$ we define the  rectangle $R_i=[r_0-\frac{1}{n}, r_0+\frac{1}{n}]\times \alpha_i$ and note that
$\mu (\alpha_i)=O(\frac{\alpha_i}{n})$. Let $B=\{ R_i : \alpha_i \in B_1\}$, then $\mu (B) \le n^{-1}n^{-1/4}=n^{-5/4}$ and so can be neglected. 
Let $G= \{ R_i  \in B^c\}$.  If $R_i \in G$ then 
$\mu (R_i)\ge n^{-3/2}$ and is of length $\ge n^{-1/2}$ in the $\vartheta$ direction and width $1/n$ in the $r$-direction. If $R_i \in G$ then $T^{[C\log n]}R_i $ is a connected
`rectangle' which has expanded in the unstable direction, contracted in the stable direction and may wind around  the phase space at most once.
 $T^{[C\log n]}R_i $ intersects  $U_n$ transversely (since $L$ is transverse to the unstable cone) in a connected component of measure 
$O(n^{-1/2} \mu (R_i))$. We estimate $\mu (U_n\cap T^{-j} (U_n))  \le \mu (R_i \in B) + \sum_{R_i \in B^c} \mu (U_n \cap T^{j} (R_i))\le C n^{-5/4} \mu (U_n) $. 

and conclude,
\[
\lim_{n\to\infty} n\sum_{j=1}^{C\log n}\mu(U_n\cap T^{-j} (U_n)) = 0.
\]

\begin{figure}
\centering
\begin{minipage}{0.49\textwidth}
\centering
\begin{tikzpicture}[scale=4]
\draw[thick] (0,0) rectangle (1,1);
\node at (0,0) [below] {$r$};
\node at (0,0) [left] {$\vartheta$};
\draw[fill=gray,opacity=0.2] (0.4,0) rectangle (0.6,0.15);
\draw[fill=gray,opacity=0.5] (0.4,0.15) rectangle (0.6,0.2);
\draw[fill=gray,opacity=0.2] (0.4,0.2) rectangle (0.6,0.35);
\draw[fill=gray,opacity=0.2] (0.4,0.35) rectangle (0.6,0.55);
\draw[pattern=north west lines] (0.4,0.35) rectangle (0.6,0.55);
\draw[fill=gray,opacity=0.5] (0.4,0.55) rectangle (0.6,0.60);
\draw[fill=gray,opacity=0.2] (0.4,0.60) rectangle (0.6,0.75);
\draw[fill=gray,opacity=0.2] (0.4,0.75) rectangle (0.6,0.85);
\draw[fill=gray,opacity=0.2] (0.4,0.85) rectangle (0.6, 1);
\draw[thick] (0.5,0)--(0.5,1);
\node at (0.5,0) [below] {$U_n$};
\draw [|-|,thick] (0.4,0)--(0.6,0);
\draw [|-|,thick] (0.4,0.35)--(0.4,0.55);
\node at (0.4, 0.45) [left] {$\alpha_i$};
\end{tikzpicture}
\caption*{(a)}
\end{minipage}
\begin{minipage}{0.49\textwidth}
\centering
\begin{tikzpicture}[scale=4]
\draw[thick] (0,0) rectangle (1,1);
\node at (0,0) [below] {$r$};
\node at (0,0) [left] {$\vartheta$};
\draw[fill=gray,opacity=0.2] (0.4,0) rectangle (0.6,0.15);
\draw[fill=gray,opacity=0.5] (0.4,0.15) rectangle (0.6,0.2);
\draw[fill=gray,opacity=0.2] (0.4,0.2) rectangle (0.6,0.35);
\draw[fill=gray,opacity=0.2] (0.4,0.35) rectangle (0.6,0.55);
\draw[pattern=north west lines,opacity=0.2] (0.4,0.35) rectangle (0.6,0.55);
\draw[fill=gray,opacity=0.5] (0.4,0.55) rectangle (0.6,0.60);
\draw[fill=gray,opacity=0.2] (0.4,0.60) rectangle (0.6,0.75);
\draw[fill=gray,opacity=0.2] (0.4,0.75) rectangle (0.6,0.85);
\draw[fill=gray,opacity=0.2] (0.4,0.85) rectangle (0.6, 1);
\draw[thick] (0.5,0)--(0.5,1);
\draw[pattern= north west lines,opacity=0.5] (0.25,0)--(0.75,1)--(0.85,1)--(0.35,0)--(0.25,0);
\node at (0.5,0) [below] {$U_n$};
\draw [|-|,thick] (0.4,0)--(0.6,0);
\draw [|-|,thick,xshift=-0.5,yshift=0.5] (0.25,0)--(0.75,1);
\draw[|-|,thick, xshift=-1, yshift=1] (0.4,0.3)--(0.6,0.7);
\end{tikzpicture}
\caption*{(b)}
\end{minipage}
\caption{(a) Intersection of points with $r=r_0$ that will not hit a extremal in $C\log n$ iterates. (b) Expansion of a single rectangle of side-length $\alpha_i$. Lines indicate portion that intersects $U_n$.}
\end{figure}

\textbf{Intermediate Returns.}

The proof of this section is similar to that for a  hyperbolic toral automorphism case but with additional complications due to the presence of discontinuities for $T$,  causing the unstable manifolds to fragment into small pieces. A scenario which needs to be ruled out is that a large number of small pieces of
fragmented  unstable manifolds may find themselves again in $U_n$. To overcome this we use the following property satisfied by the planar dispersing billiard map:
\par\indent\textbf{One-step expansion.} For $\alpha \in (0,1]$,
\[
\lim_{\delta\to 0}\inf\sup_{W:|W|<\delta}\sum_n \Big(\frac{|W|}{|V_n|}\Big)^{\alpha}\cdot\frac{|T^{-1}V_n|}{|W|}<1,
\]
where the supremum is taken over regular unstable curves $W\subset X$, $|W|$ denotes the length of $W$, and $V_n$, $n\ge 1$, the smooth components of $T(W)$, $\alpha\in (0,1]$. The class of regular curves includes our local unstable manifolds~\cite{CM}.

The expansion by $DT$ is unbounded  and this may lead to  different expansion rates at different points on $W^u (x)$. To overcome this
effect and obtain uniform estimates on the densities of conditional SRB measure it is common to define homogeneous local
unstable and local stable manifolds. This is the approach adopted in~\cite{BSC1,BSC2,CM07,Y98}. Fix a large $k_0$ and 
define for $k>k_0$
\[
I_k=\{(r,\vartheta): \frac{\pi}{2} -k^{-2} <\vartheta<\frac{\pi}{2}-(k+1)^{-2} \},
\]
\[
I_{-k}=\{(r,\vartheta): -\frac{\pi}{2} +(k+1)^{-2} <\vartheta<-\frac{\pi}{2}+k^{-2} \},
\]
and 
\[
I_{k_0}=\{(r,\vartheta): -\frac{\pi}{2} +k_0^{-2} <\vartheta<\frac{\pi}{2}- k_0^{-2} \}.
\]
In our setting we call a local unstable (stable) manifold $W^u (x)$, ($W^s (x)$) homogeneous if  for all $n\ge 0$
$T^n  W^u (x)$ 
($T^{-n}  W^s (x)$) does not intersect any of the line segments in $\cup_{k>k_0} (I_k\cup I_{-k})\cup I_{k_0}$. Homogeneous
$W^u (x)$ have almost constant  conditional SRB densities $\frac{d\mu_x}{dm_x}$ in the sense that 
there exists $C>0$ such that $\frac{1}{C} \le \frac{d\mu_x (z_1)}{dm_x} /\frac{d\mu_x (z_2)}{dm_x} \le C$ for all $z_1,~z_2 \in W^u (x)$ (see ~\cite[Section 2]{CM} and the remarks following Theorem 3.1).

From this point on all the local unstable (stable) manifolds that we consider will be homogeneous. 
We may as well suppose all such curves are contained in $R_i \in G$ as
$\mu (B)< n^{-5/4}$.

We now take care of the times $[C\log n] <j< (\log n)^{1+\delta}$. If $W^u (x) \cap U_n \subset R_i \in B^c$  then $T^{[C\log n]}$ has expanded $W^u (x)$ by a factor
$\Lambda^{C\log n}=n^{C\log \Lambda}=n^{\beta}$ for some $\beta >0$ and the iterates of the components of $W^u(x)\cap U_n$ 
have not  hit a extremal set in the first $[C\log n]$ iterates. Let $\gamma_n(x)=W^u (x)\cap U_n$. 
By~\cite[Theorem 5.7]{CM07}  $\mu ( W^u (x) < n^{-1-\beta/2}) < n^{-1-\beta/2}$ so we may require all $W^u (x) \in \cup_{R_i \in G} R_i$ to satisfy $|\gamma_n (x) | > n^{-1-\beta/2}$.

 Now we consider $\mu(U_n \cap T^{-j}(U_n ))$ for $C\log n\le j\le (\log n)^{1+\delta}$. Note that $T^j(\gamma_n (x) )$ consists 
of a connected curve  for $j\le C\log n$.  Recall by expansion under the map we have $|T^j \gamma_n (x)|\ge n^{\beta}|\gamma_n (x)|>n^{-1+\beta/2}$. If we iterate this component further such that
 $T^{i+j}\gamma_n(x)$, $i>0 $ intersects  a extremal line then we may decompose $T^{i+j}\gamma_n (x)$ into smooth connected components $V_n$ and their preimages $Y_n\subset T^j \gamma_n (x) $ so that $T^i$ maps $Y_n$ onto $V_n$ diffeomorphically and with uniformly bounded distortion. Applying one-step expansion for $p\in \gamma_n (x)$ gives,
\[
\sum_n \Big (\frac{|\gamma_n (x)|}{|V_n(p)|}\Big)^{\alpha}\Big|\frac{Y_n(p)}{\gamma_n (x)}\Big |<1.
\]
 Fix $T^j \gamma_n(x)$ and for every point $p\in T^j \gamma_n(x)$ let $d\mu_{\gamma}(p) = \frac{|Y_n(p)|}{|\gamma(x)|}$ be the density of a probability measure $\mu_{\gamma}(p)$ on $T^j \gamma_n(x)$ and $f(p) = \big(\frac{|\gamma_n (x)|}{|V_n(p)|}\big)^{\alpha}$ a function on this probability space. Now $\{p\in T^j \gamma_n (x) : |V_n(p)|<n^{-1+\varepsilon\beta/2} \} \subset\{p\in T^j \gamma_n (x):f(p)>n^{(1-\varepsilon) \beta/2\alpha}\}$ and by Markov's inequality $\mu_{\gamma}\{p \in T^j  \gamma_n (x): |V_n(p)|<n^{-(1+\varepsilon)\beta/2} \}\le n^{-(1-\varepsilon)\beta/2\alpha}$.

 We choose $\varepsilon$ sufficiently small so that $\rho_1 := 1-(1-\varepsilon)\beta/2\alpha>0$ (since $\beta<1$) and define $\rho = \min\{\rho_1,\beta/2\alpha\varepsilon\}$. With our choice of $\varepsilon$, if $|V_n|\ge n^{-1+\epsilon\beta/2}$ then,
\[
\frac{|V_n\cap U_n |}{|V_n|}\le C_1 n^{-\rho}.
\]
By bounded distortion of the map $T$, after throwing away the $V_n$ such that $|V_n|\le n^{-1+\varepsilon\beta/2}$ we have 
\[
\frac{|T^i \gamma_n (x) \cap U_n )|}{|\gamma_n (x) |}\le C_2 n^{-\rho}.
\]
and by bounded distortion again we have,
\[
\frac{|\gamma_n (x) \cap T^{-i}(U_n)|}{|\gamma_n (x) |}\le C_3 n^{-\rho}.
\]
\par This provides a bound on the length of the intersection of a single unstable manifold $\gamma_n (x) $.
We may now use the fact that $\mu$ decomposes as a product measure on $U_n$ so that if we consider all manifolds of $R_i\in G$ we have,
\[
\mu( U_n \cap T^{-j}(U_n))\le C_4 n^{-1-\rho}.
\]

Putting these results together implies,
\[
\lim_{n\to\infty} n \sum_{j=C\log n}^{(\log n)^{1+\delta}}\mu(U_n\cap T^{-j}(U_n)) = 0.
\]
Condition $\DD^{'}(u_n)$ follows.
\begin{remark}

Using essentially the same analysis it is standard to show that the return time statistics to $L=\{ (r,\vartheta): r=r_0\}$ is standard simple Poisson. To see this we need verify
condition $D^{*}_q (u_n)$ of ~\cite[Section 2]{CNZ}, but the proof of this is a minor modification of $\DD(u_n)$. In contrast suppose $(r_0,\vartheta_0)$ is a periodic
point of period $q$, then we would obtain a compound Poisson process as given in~\cite[Theorem 2]{CNZ}.

\end{remark}


\subsection{Proof of Theorem~\ref{thm:coupled} and Theorem~\ref{thm:block}.}

We give the proof in detail only for the case of two coupled maps, as the proofs in the other cases 
are the same with obvious modifications. The uniform expansion away from the invariant subspace  plays the same role in each setting. Note that the subspace $L$ of Theorem~\ref{thm:block} is invariant, and we will show that there is 
uniform expansion in the directions orthogonal to  $L$.

Recall $\phi (x,y)=-\log |x-y|$, a function maximized on the line segment or circle $L=\{(x,y):y=x\}$.  For $\tau>0$ 
define $u_n(\tau)$ by $n\mu (\phi >u_n (\tau))=\tau$, and $U_n=\{\phi >u_n (\tau)\}$. Define $A_n=\{\phi > u_n,
\phi\circ F <u_n\}$ and recall for a set $B$, $\mathscr{W}_{s,l}(B)=\bigcap_{i=s}^{s+l-1} F^{-i}(B)$.

Note that the invariant line $L$ is uniformly repelling in  the orthogonal direction $(1,-1)$ since writing $y-x=\epsilon$ we see $\epsilon \rightarrow (1-\gamma)[T(x+\epsilon)-T x]\sim (1-\gamma) DT(x) \epsilon+O(\epsilon^2)$ under the map $F$. 

Furthermore $A_n$ is a union of two rectangles and $A_n\cap F^{-2} A_n=\emptyset$ as a result of uniform expansion away from the invariant line $L$.

Condition $\DD(u_n)$ follows easily by an approximation argument using exponential decay of correlations of 
Lipschitz versus $L^{\infty}$ functions taking  $t_n=(\log n)^5$ say.

Now we prove condition $\DD'_q(u_n)$ (for $q=1$), namely
\[
\lim_{n\rightarrow\infty}\,n\sum_{j=1}^{\lfloor n/k_n\rfloor}\mu\left( A_n\cap F^{-j}\left(A_n\right)
\right)=0.
\]

Note that by uniform repulsion from the invariant line $L$ there exists $C_4$ such that
 for $j=1,\ldots, C_4 \log n$, $\mu ( A_n\cap F^{-j} A_n )=0$. This follows
 since $F^{-1} A_n \cap A_n=\emptyset$ (by definition) and uniform repulsion from the invariant line
 ensures also  $F^{-j} A_n \cap A_n=\emptyset$ for a certain number of iterates $j=1,\ldots, C_4 \log n$ until for all $(x,y)$ in $A_n$,
 $|F^j (x,y)|=O(1)$ (i.e. until the expansion in the $L^{\perp}$ direction is $O(n)$).

  As $DF$ is bounded and uniformly expanding, in all directions $A_n$ has been expanded by the map $F^{[C_4\log n]}$ by at least $n^{\alpha}$ for some $0<\alpha<1$. To see this, note that for any expanding map  expansion of $A_n$ by the map $F^{[C_4\log n]}$ given by at least $C_5 |DT|_{min}^{C_4\log n}\sim n^\alpha$.

\begin{figure}
\begin{tikzpicture}[scale=4]
\draw (0,0)--(0,1)--(1,1)--(1,0)--(0,0);
\draw[thick] (0,0)--(1,1);
\draw[fill=gray,opacity=0.2] (0,0.15)--(0,0.25)--(0.75 ,1)--(0.85,1);
\draw[fill=gray,opacity=0.2] (0,0.15)--(0,0.55)--(0.45 ,1)--(0.85,1);
\draw[fill=gray,opacity=0.2] (0.15,0)--(0.25,0)--(1,0.75)--(1,0.85);
\draw[fill=gray,opacity=0.2] (0.15,0)--(0.55,0)--(1,0.45)--(1,0.85);
\draw[opacity=0.2] (0,0.15)--(0,0.25)--(0.75 ,1)--(0.85,1);
\draw[opacity=0.2] (0.15,0)--(0.25,0)--(1,0.75)--(1,0.85);
\draw[dotted] (0,1)--(1,0);
\draw[->,>=stealth] (0.5,0.5)--(0.6,0.6);
\draw[->,>=stealth] (0.5,0.5)--(0.4,0.4);
\draw[->,>=stealth] (0.5,0.5)--(0.6,0.4);
\draw[->,>=stealth] (0.5,0.5)--(0.4,0.6);
\node at (1,0)[below]{$x$};
\node at (0,1)[left]{$y$};
\node at (0.75,1)[below]{\small$x=y$};
\draw[|-|] (1,0.85)--(1,0.75);
\node at (1,0.8) [right]{$O(\frac{1}{n})$};
\end{tikzpicture}
\caption{Expansion of the set $A_n$ (given in gray) under the map $F$. The thick, black line represents the line of maximization. Arrows indicate uniform expansion under $F$ in all directions.}
\end{figure}
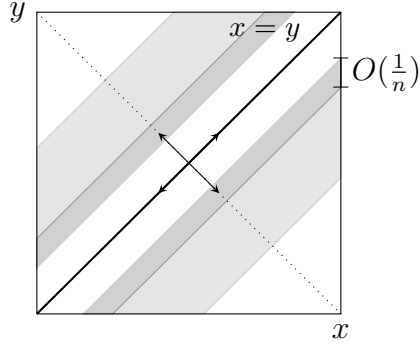

Choose $C_3\ge C_4$ large enough that $\mu ( A_n\cap F^{-j}(A_n ))\le \frac{1}{n^{3/2}}$, this is possible by exponential decay of correlations and a Lipschitz approximation to $1_{A_n}$.

Thus for $C_4 \log n \le j \le C_3 \log n$, $\mu ( A_n\cap F^{-j} (A_n) )\le \frac{1}{n^{1+\alpha}}$.  For $1\le j \le C_4 \log n$, $\mu (F^{-j} (A_n) \cap A_n)=0$ for  $C_4 \log n \le j \le C_3 \log n$, $\mu ( A_n\cap F^{-j} A_n )\le \frac{1}{n^{1+\alpha}}$ and for $j\ge C_4 \log n$, $\mu ( A_n\cap F^{-j} (A_n ))\le \frac{1}{n^{3/2}}$.

This implies $\DD'_q(u_n)$ for $q=1$ (corresponding to the fact that $L$ is fixed).

Finally we compute the extremal index, changing coordinates to $v=\frac{x-y}{\sqrt{2}}$, $u=\frac{x+y}{\sqrt{2}}$ we have

\[
\theta=\lim_{n\to\infty}\theta_n=\lim_{n\to\infty}\frac{\mu(A_n)}{\mu(U_n)}.
\]
However
\[
\lim_{n\to \infty}\frac{\mu(A_n)}{\mu(U_n)}
= \lim_{n\to \infty} [1- \int_{0}^{\frac{1}{n[Tv]}} \int_{0}^{\frac{1}{n}}\tilde{h}(u,v)dudv] / \int_{0}^{\frac{1}{n}} \int_{0}^{\frac{1}{n}}\tilde{h}(u,v)du dv. \]
Suppose $m(U_n)=O(\epsilon^{1/4})$, $\epsilon <\epsilon_0$.  Since $|\tilde{h}|_{\alpha}<\infty$, $m(x\in U_n : {\it osc} (\tilde{h}, B_{\epsilon} (x)) >\sqrt{\epsilon})<\sqrt{\epsilon}=O(m(U_n)^2)$.  We may assume that $\tilde{h}$ is regularized along the diagonal in the sense that for Lebesgue almost every $u$, $\tilde{h}(u,u)$ is the average of the limits of $\tilde{h}(u,v)$ and $\tilde{h}(u,-v)$ as $v\to 0$.
Thus, as expansion along $v$ at $v=0$ is given by $(1-\gamma)DT(u)$, and $\tilde{h}$ is essentially bounded
\[
\theta=1- \int_L \frac{\tilde{h}(u,u)}{(1-\gamma)|DT(u)|} du.
\]



\begin{remark}
Our techniques allow us to obtain similar results to that of ~\cite{sandro_coupled} in a simpler setting through a pure probabilistic approach and extend these results to blocks of synchronization discussion in ~\cite[Section 7.2]{sandro_coupled}.
\end{remark}

\subsection{Numerical Results for the Extremal Index}\label{sec.numerics}
In this section we provide numerical estimates for the extremal index to support the theoretical results for the coupled uniformly expanding map and the hyperbolic toral automorphism provided in Theorems \ref{thm:coupled} and \ref{thm:block} and Theorem \ref{thm.anosov}, respectively. We begin by verifying that the numerical estimates we obtain from the coupled systems agree with that of \cite{sandro_coupled}. Then, we extend these results to include estimates for the extremal index over blocks of synchronization where each block introduces a new invariant direction and changes the value of the extremal index. We end with a numerical investigation for Arnold's cat map as an example of a hyperbolic toral automorphism where the alignment of the singlarity set taken as a line $L$ in the space and existence of periodic orbits along $L$ determine the value of the extremal index.
\subsection{Coupled systems of uniformly expanding maps}
Numerical barriers in computing trajectories in piecewise uniformly expanding maps are given by the fact that
\begin{itemize}
\item[(i)]The periodic orbits are dense making long trajectories not easily computable.
\item [(ii)]Round off errors may produce unreliable results.
\end{itemize}
To overcome (i) we employ a numerical technique adapted from \cite{V.et.al} to prevent trapping of the orbit near the fixed point by adding a small $\varepsilon=O(10^{-2})$ amount to the trajectory. Arguments for this technique are typically given in the form of a shadowing lemma which states the existence of a true orbit that is $\epsilon$-close to the computed orbit; we will support this argument through a more numerical approach. We first note that \cite{FFLTV} proves the existence of an EVL for randomly perturbed piecewise expanding maps provided this perturbation $\varepsilon>10^{-4}$. Futher, \cite{sandro_coupled} provides evidence that the extremal index is qualitatively robust under small $\varepsilon=10^{-2}$ additive noise. To overcome (ii), in light of our long trajectories ($t=10^6$), we refer to \cite{FMT} where the round off error resulting from double precision computation was shown to be equivalent to the addition of random noise of order $10^{-7}$.

\subsubsection*{Estimating the EI for the coupled map system over the whole extremal set.}
We estimate the extremal index in a similar way to that of~\cite{sandro_coupled} for $\phi(\bar{x}) = -\log(|| p^{\perp}||)$ using the formula provided by S\"{u}veges~\cite{suveges}. The code for this estimate can be found in~\cite{V.et.al}. From Theorem~\ref{thm:coupled} we expect,
\[
\theta = 1-\frac{1}{(1-\gamma)^{m-1}}\frac{1}{|DT|^{m-1}}.
\]
We compute the extremal index for fixed $m=2$ and varying values of $\gamma$, and varying values of both $\gamma$ and $m$. Our results coincide with that of \cite{sandro_coupled}; higher values of $m$ and lower values of $\gamma$ produce an extremal index near 1. Low values of $\gamma$ give higher weights to the non-coupled components of the map resulting in a system which behaves more independently. Lower values of $m$ result in a more dependent system since the coupled term is more affected by changes while larger values of $m$ result in a coupled term which is averaged over a larger number of maps and less affected by individual changes. For results see Figure~\ref{fig:ei}.

\begin{figure}[h!]
\begin{minipage}{0.49\textwidth}
\includegraphics[width=\textwidth]{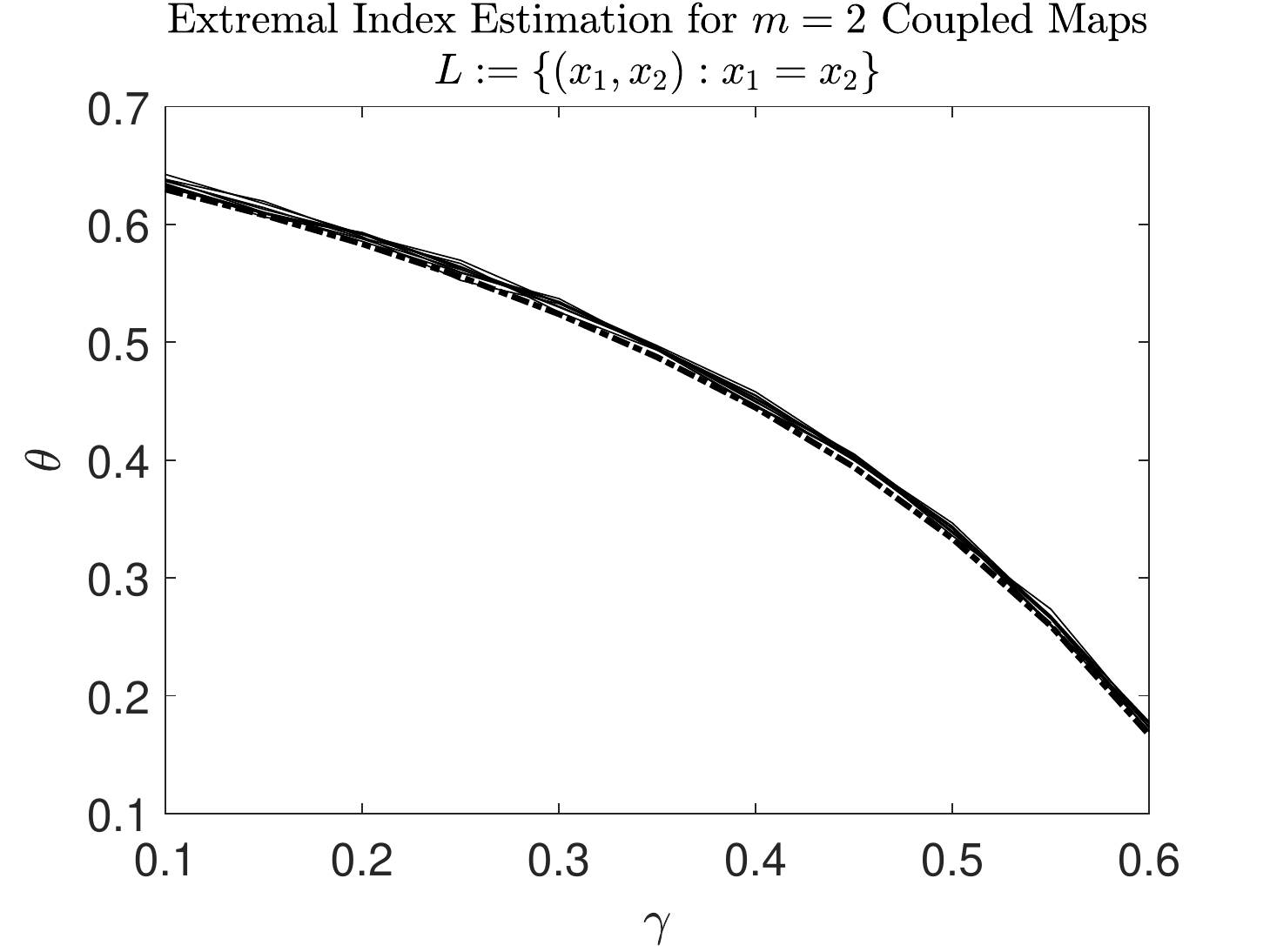}
\caption*{(a)}
\end{minipage}
\begin{minipage}{0.49\textwidth}
\includegraphics[width=\textwidth]{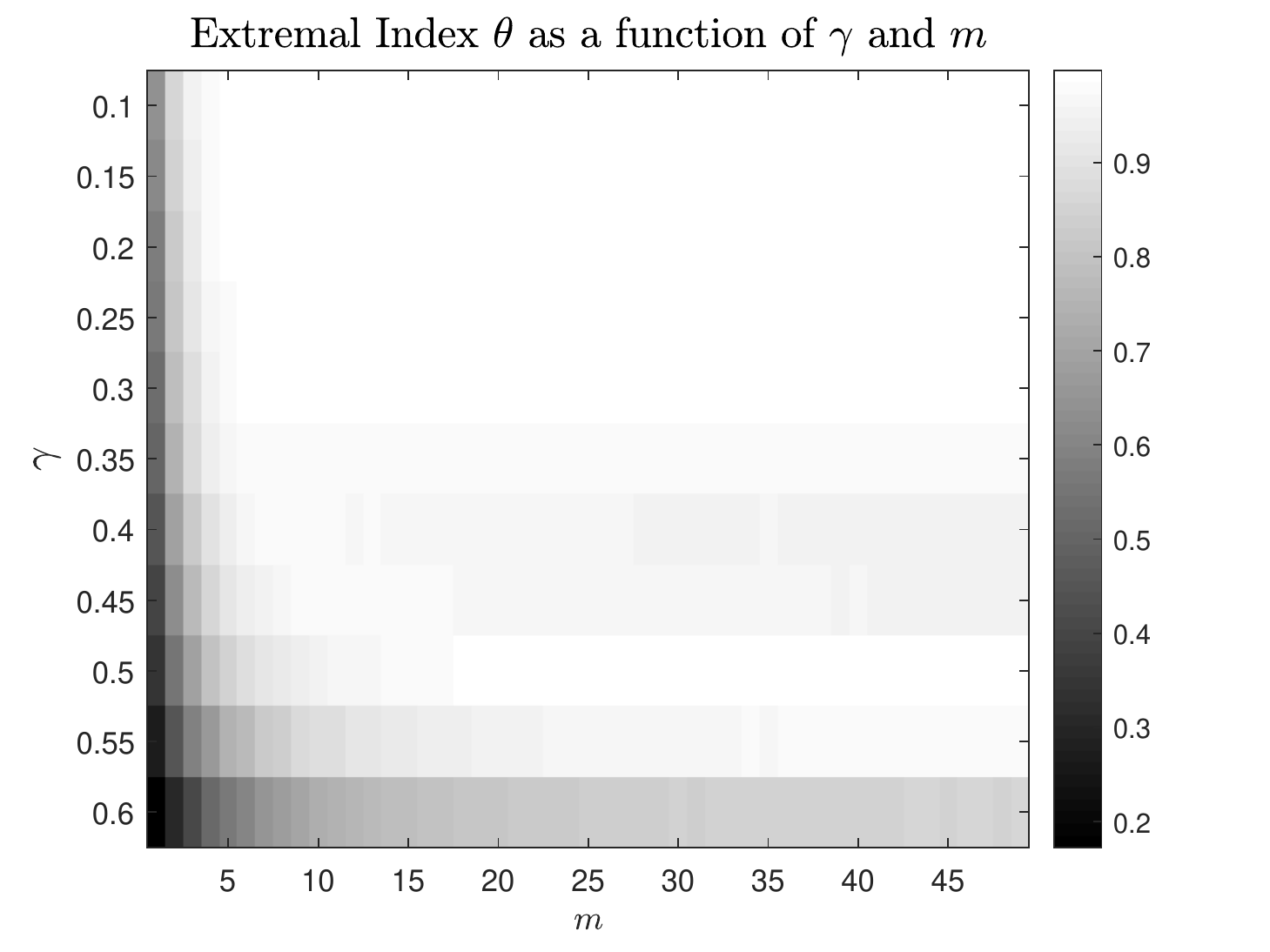}
\caption*{(b)}
\end{minipage}
\caption{\label{fig:ei} Extremal index $\theta$ estimation for the $m$-coupled map $F$ with $\phi(x) = -\log (||p^{\perp}||)$ where the set of maximization $L=\{(x_1,x_2,\dots,x_m):x_1=x_2=\dots=x_{m}\}$ for (a) fixed $m$ and varying $\gamma$ (10 different realizations $t = 10^6$) and (b) varying $m$ and $\gamma$. The marked line indicates the theoretical value of $\theta$ given.}
\end{figure}

\subsubsection*{Estimating the EI for the coupled map system over blocks of synchronization.}
\par We provide numerical estimates of the extremal index in a more specific setting of block synchronization where $L=\{(x_1,x_2,\dots,x_m):x_1=x_2=\dots=x_{m}\}$ and $L=\{(x_1,x_2,\dots,x_m):x_1=x_2=\dots=x_{m-1},x_m\}$. From Theorem~\ref{thm:block} we expect,
\[
\theta=1-\frac{1}{(1-\gamma)^{m-1}}\frac{1}{|DT|^{m-1}}.
\]
for $L=\{(x_1,x_2,\dots,x_m):x_1=x_2=\dots=x_{m}\}$ and,
\[
\theta=1-\frac{1}{(1-\gamma)^{m-2}}\frac{1}{|DT|^{m-2}}
\]
for $L=\{(x_1,x_2,\dots,x_m):x_1=x_2=\dots=x_{m-1},x_m\}$. Defining $L$ in this way reduces the spacial dimension in which expansion away from $L$ can occur. This results in a extremal index equivalent to that of an $m-1$ coupled system. We give results in the case when $m=5$ (see Figure~\ref{fig:bs2}) . 
\begin{figure}[h!]
\begin{minipage}{0.49\textwidth}
\includegraphics[width=\textwidth]{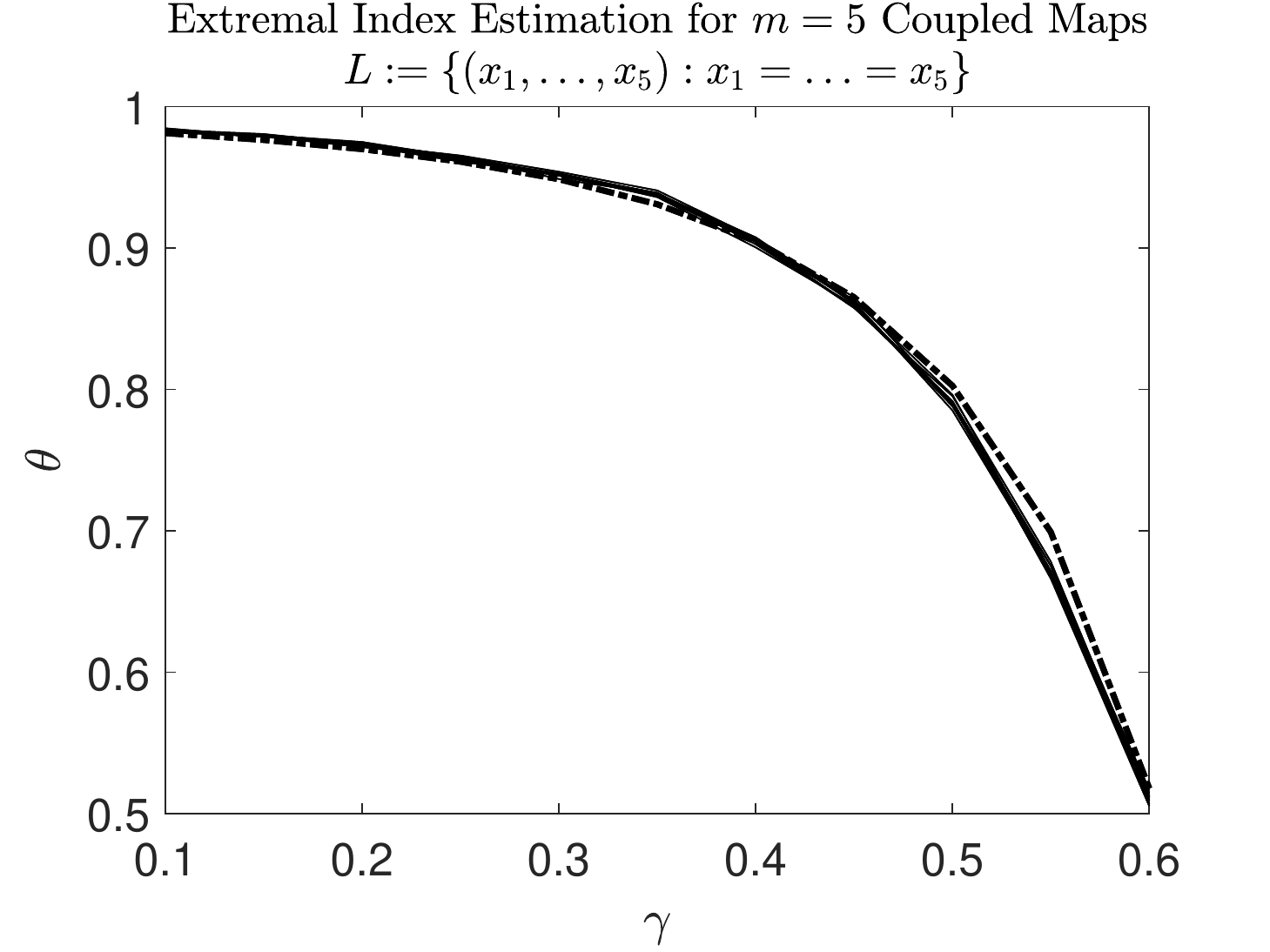}
\caption*{(a)}
\end{minipage}
\begin{minipage}{0.49\textwidth}
\includegraphics[width=\textwidth]{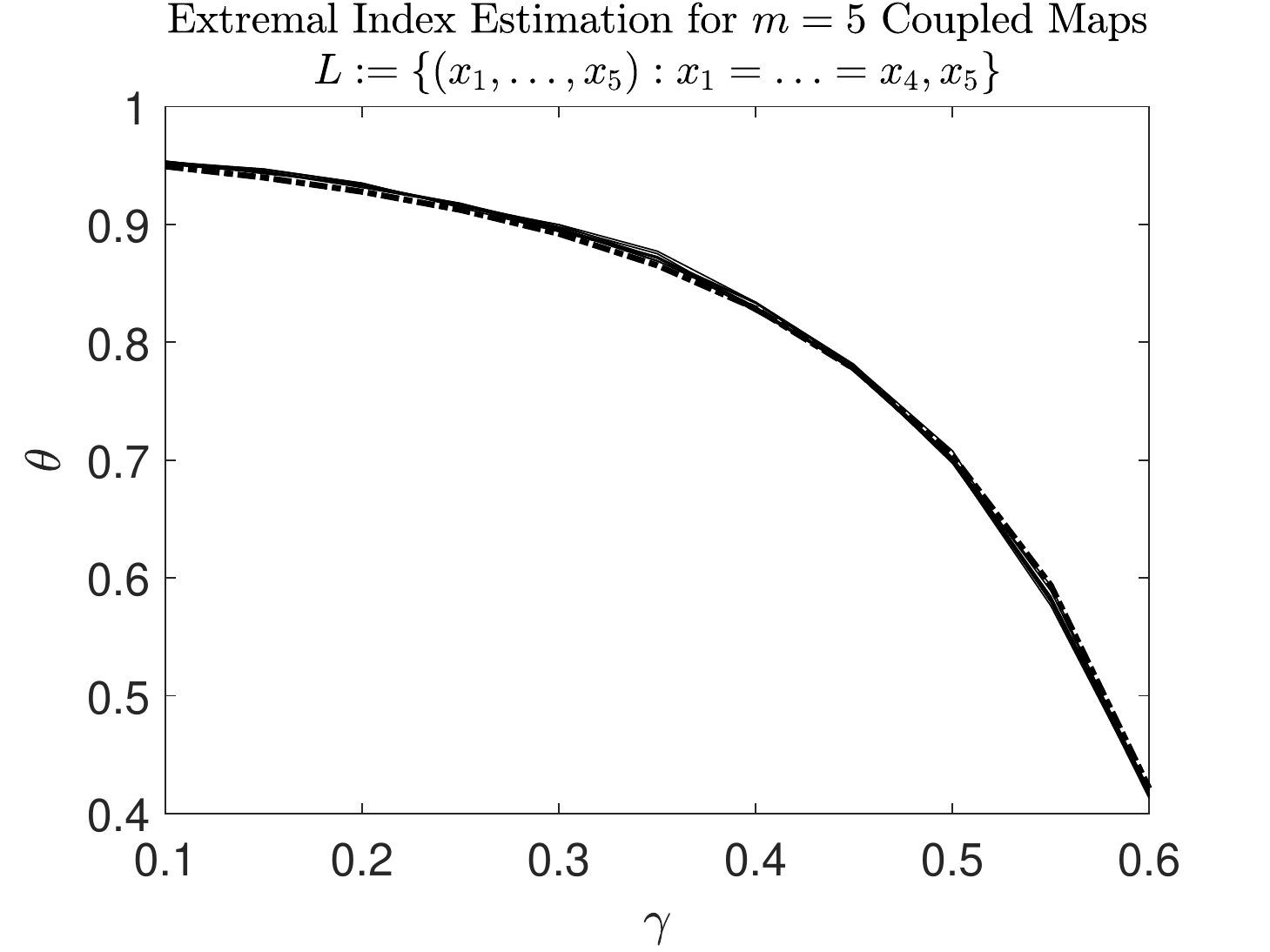}
\caption*{(b)}
\end{minipage}
\caption{\label{fig:bs2} Extremal index $\theta$ estimation (10 different realizations, $t=10^6$) for the $m$-coupled map $F$ with $\phi(x) = -\log(||p^{\perp}||)$ where (a) $L$ is the line $x_1=x_2=\dots=x_m$ and (b) $L$ is the plane $x_1=x_2=\dots=x_{m-1},x_m$. The marked line indicates the theoretical value of $\theta$ given.}
\end{figure}

We also consider blocks of successive indices in the general setting of block synchronization so that $L$ can be defined as any combination of block sequences. From Theorem~\ref{thm:block} we expect the value of the extremal index to be determined by the spacial dimension of expansion for the system. In the following numerical examples we consider $m=5$ and note that the extremal index for that of $L=\{(x_1,\dots,x_5):x_1=x_2=x_3=x_4,x_5\}$ is equivalent to that of $L = \{(x_1,\dots,x_5):x_1=x_2=x_3,x_4=x_5\}$. This is expected since they share the same number of non-invariant directions of expansion.

\begin{figure}[h!]
\begin{minipage}{0.49\textwidth}
\includegraphics[width=\textwidth]{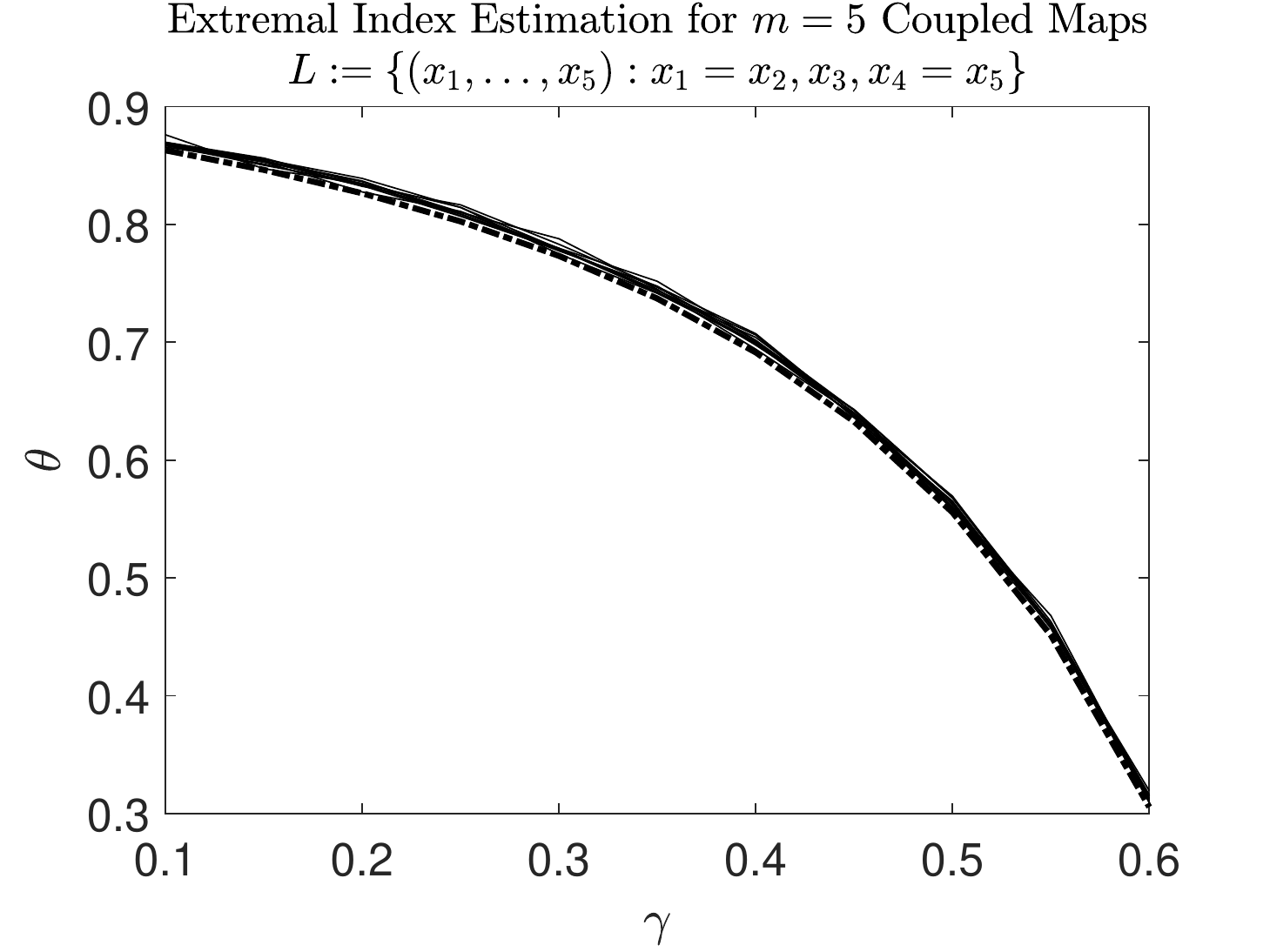}
\caption*{(a)}
\end{minipage}
\begin{minipage}{0.49\textwidth}
\includegraphics[width=\textwidth]{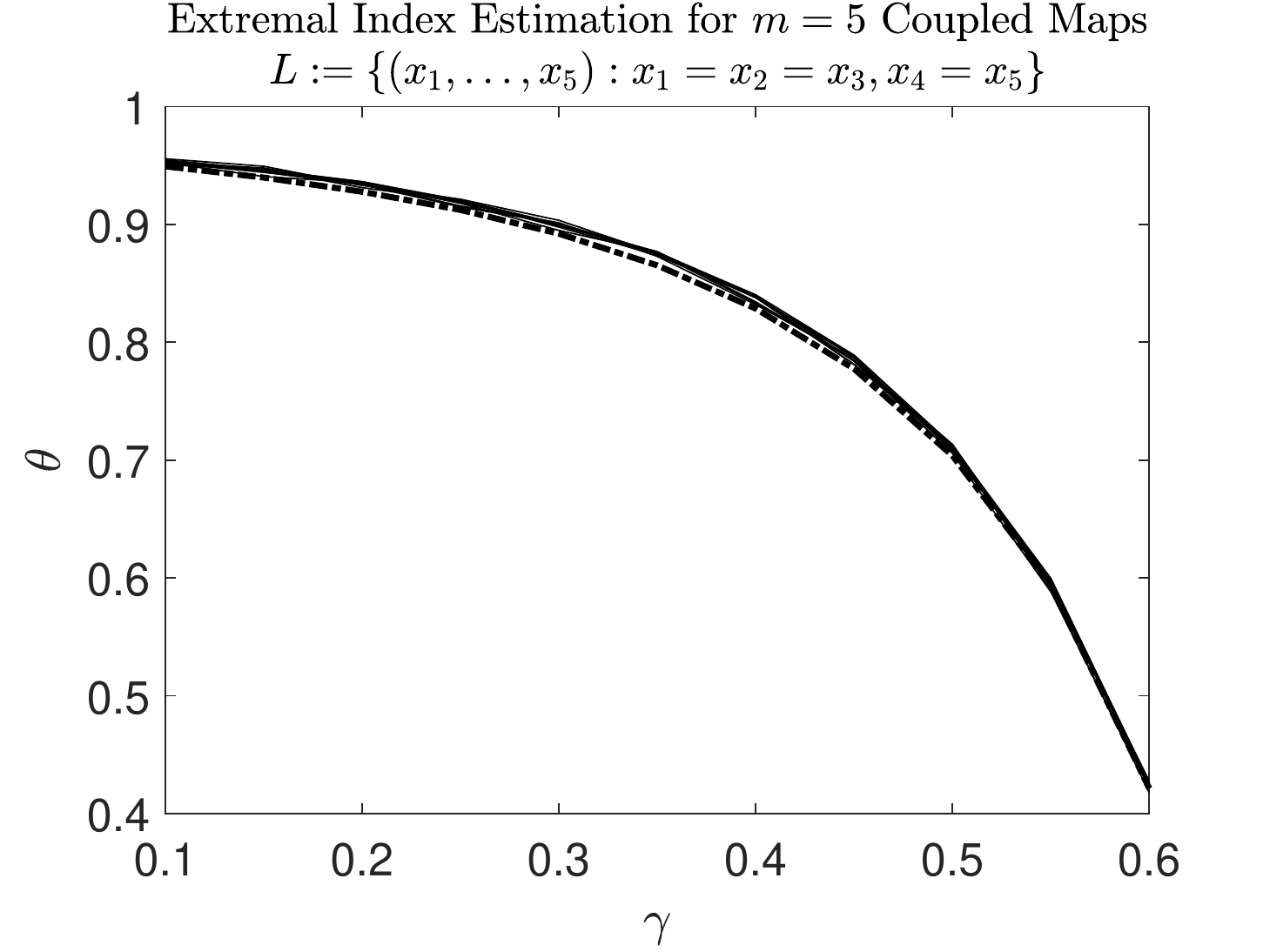}
\caption*{(b)}
\end{minipage}
\caption{\label{fig:bs} Extremal index $\theta$ estimation (10 different realizations, $t=10^6$) for the $m$-coupled map $F$ with $\phi(x) = -\log(||p^{\perp}||)$ where (a) $L$ is the set of two planes $x_1=x_2$ and $x_4=x_5$ so that $\theta = 1-\frac{1}{(1-\gamma)^{2}|DT|^2}$ (b) $L$ is the set of planes $x_1=x_2=x_3$ and $x_4=x_5$, $\theta = 1-\frac{1}{(1-\gamma)^{3}|DT|^3}$. The marked line indicates the theoretical value of $\theta$ given.}
\end{figure}

\subsection{Hyperbolic toral automorphisms.}
We compute trajectories for increasing time intervals of Arnold's cat map given by,
\[
T\begin{pmatrix} x_1\\ x_2 \end{pmatrix} = \begin{pmatrix}2 & 1\\ 1 & 1\end{pmatrix}\quad\begin{pmatrix} x_1\\ x_2 \end{pmatrix} \mod 1.
\]
The uniformly hyperbolic structure of this map allows us to calculate long trajectories without the risk of points being trapped in a few time steps. The stability of this map ensures that the qualitative behavior is unaffected by small perturbations. We use this to argue the accuracy of the calculated orbit up to $t=10^4$ under double precision. 

From Theorem~\ref{thm.anosov} we expect the value of the extremal index $\theta$ to depend on both the alignment of $v$ in the observable $\phi(x) = -\log d(x,L)$, with $d$ the usual Euclidean metric, and the existence of a periodic orbit along $L$. Figure~\ref{fig:cm1} (a) shows the extremal index estimation given by~\cite{suveges} for 10 different initial values where $v$ aligns with the unstable direction and contains a 2-periodic point. Hence, $\theta = 1-\frac{1}{\lambda^2}$. Figure~\ref{fig:cm1}(b) shows the extremal index estimation for 10 different initial values where $v$ is not aligned with the stable or unstable direction. In this setting we expect $\theta=1$. The variation from the expected value for each realization is at most $O(10^{-2})$.

\begin{figure}[h!]
\begin{minipage}{0.49\textwidth}
\includegraphics[width=\textwidth]{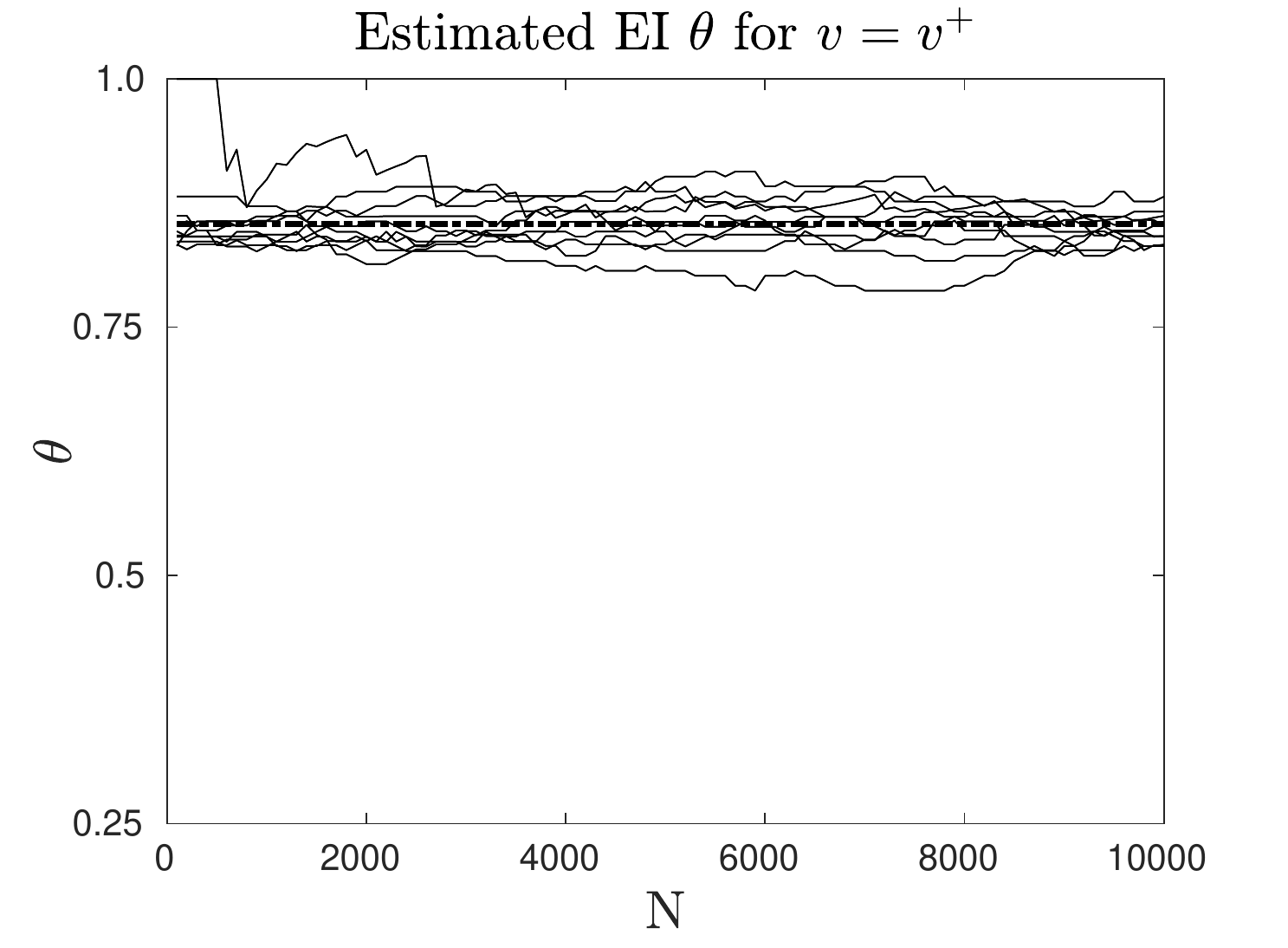}
\caption*{(a)}
\end{minipage}
\begin{minipage}{0.49\textwidth}
\includegraphics[width=\textwidth]{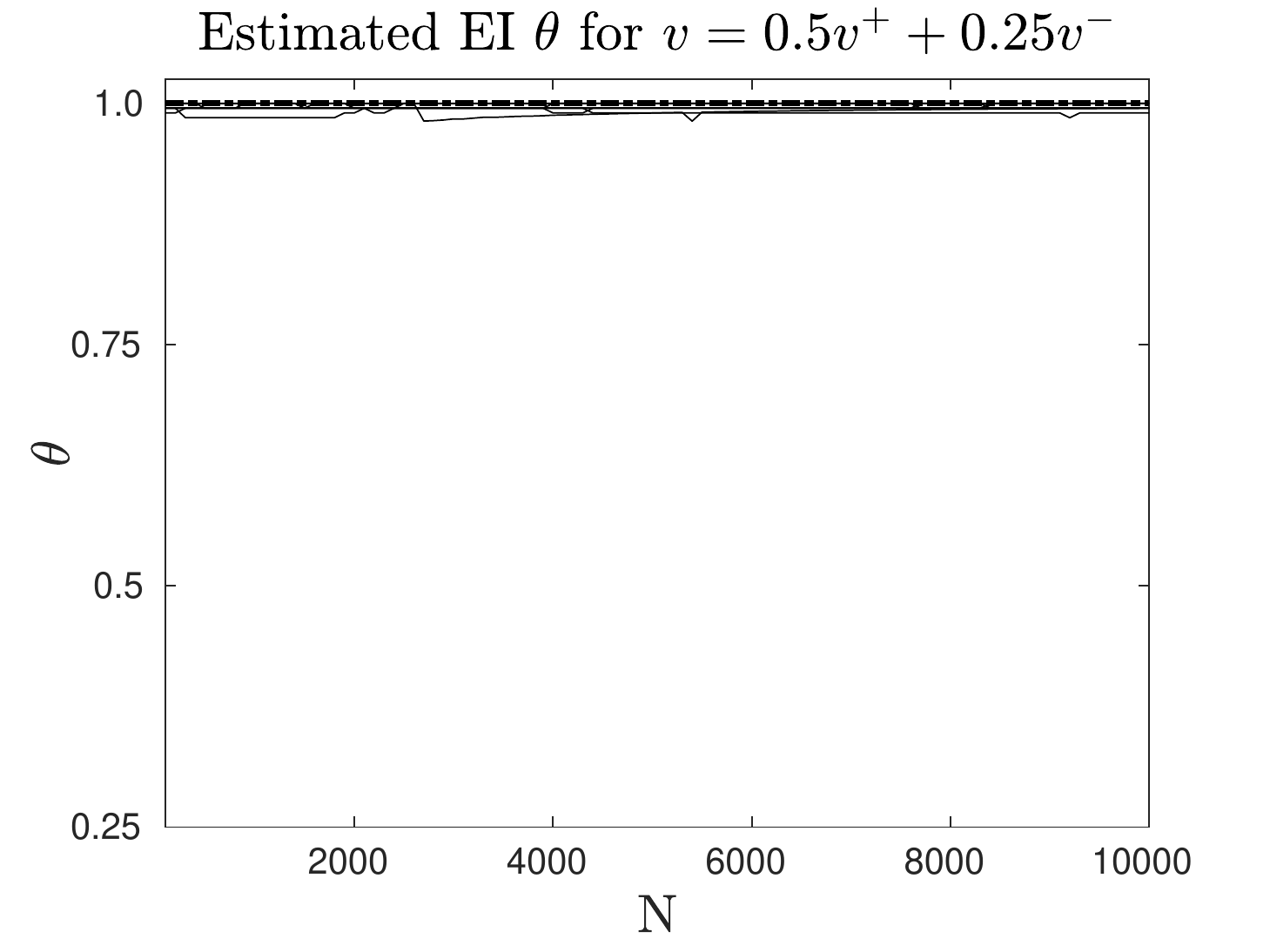}
\caption*{(b)}
\end{minipage}
\caption{\label{fig:cm1} Extremal index $\theta$ estimation (10 different realizations, $t=10^4$) for Arnold's cat map with $\varphi(x) = -\log d(x,L)$ where (a) $L=v^+$ and (b) $L=0.5v^++0.25v^-$. The marked line indicates the theoretical value of $\theta$ given.}
\end{figure}

\section{Discussion: towards more general observables and non-uniformly hyperbolic systems}\label{sec.discussion}
In this article we have focused on hyperbolic systems and considered observables whose level sets $\mathcal{S}_{\epsilon}$
shrink to a non-trivial extremal set $\mathcal{S}$, such as a line segment. We recall that
$\mathcal{S}_{\epsilon}=\{x\in X:\,d_H(x,\mathcal{S})\leq\epsilon\}$, and $d_H(x, \mathcal{S})$ is the Hausdorff distance
from $x$ to $\mathcal{S}$. Thus if $\mathcal{S}$ is a smooth curve, then for this metric $d_H$ we see that 
$\mathcal{S}_{\epsilon}$ is a thin tube of width $\epsilon$ around $\mathcal{S}$.
The observable $\phi:X\to\mathbb{R}$ we have assumed to be given by $\phi(x)=f(d_{H}(x,\mathcal{S}))$, for
some smooth function $f:[0,\infty)\to\mathbb{R}$, maximised at 0, e.g. $f(u)=-\log u$. 

As explained in Section \ref{sec.statement}, our methods extend to cases where $\mathcal{S}$ is a smooth 
curve, assuming some transversality conditions of $\mathcal{S}$ relative to the global stable/unstable manifolds of the system.
We have also considered seemingly non-generic geometrical cases, e.g. where $\mathcal{S}$ 
aligns precisely with the global stable/unstable manifolds. For hyperbolic toral automorphisms, we established the limit laws that arise in these
scenarios. More generally, it is natural to consider observables whose extremal set $\mathcal{S}$ is no longer (strictly) transverse
to the global stable/unstable manifolds, i.e. there exist points of tangency between $\mathcal{S}$ and the global manifolds. 

For the systems we have considered, the ergodic invariant measures are  absolutely continuous with
respect to the ambient (two dimensional Lebesgue) measure. For (non-uniformly) hyperbolic systems $(f,\Lambda,\mu)$ where $\Lambda$ is an
attractor the  Sinai-Ruelle-Bowen (SRB) measure $\mu$ may not be equivalent to Lebesgue. These systems include Axiom A systems, or H\'enon-like attractors
whose statistical properties (such as mixing rates) are established in \cite{Y98}. As outlined in 
Section \ref{sec.background}, there is an established literature on extreme value theory in the non-uniformly hyperbolic setting   for observables whose extremal set $\mathcal{S}$ is a point. Recently some progress has been made on more complicated geometries for $\mathcal{S}$~\cite{Haydn_Vaienti} but in a very axiomatic way.
In the case where $\mathcal{S}$ is a  line (or in higher dimensions a planar set), then we expect $\mathcal{S}$ 
to (generically) intersect a fractal attractor $\Lambda$ in a Cantor-like set. For such a set, there are various difficulties that arise
when trying to find the limit extreme value distribution distribution, in the sense of establishing \eqref{eq.ev-law}, or in particular
the limit law given by \eqref{eq.gevlimit}.
If we suspect that a limit law of the form given in equation \eqref{eq.ev-law} is going to exist, then finding the scaling
sequence $u_n$ is a first problem. For a specified observable $\phi$ (i.e. through specifying $f$), the properties of the sequence $u_n$ depend on the asymptotic properties of $\mu(\mathcal{S}_{\epsilon})$ as $\epsilon\to 0$. To estimate this measure, we cannot use local dimension estimates, and finer arguments are required based on the geometric properties of $\mu$. Furthermore, existence of a GEV limit
of the form \eqref{eq.gevlimit} is not guaranteed, as this requires $\mu(\mathcal{S}_{\epsilon})$ to satisfy conditions of regular variation in $\epsilon$ (as $\epsilon\to 0$), see \cite[Chapter 3]{V.et.al}. Axiomatic approaches, e.g. \cite{CC, Haydn_Vaienti, HRS} suggest that once we've found these 
scaling laws then an extreme value law holds in the sense of equation \eqref{eq.ev-law}. However, verification of these axioms 
still requires fine analysis. This includes verification of axiomatic conditions involving transversality of $\mathcal{S}$ with $\Lambda$,
and conditions involving how $\mu$ behaves on certain shrinking sets (such as thin annuli) on a case-by-case basis. 



\clearpage
\begin{thebibliography}{99}


\bibitem{Bowen} R. Bowen. Equilibrium states and the ergodic theory of Anosov diffeomorphisms. Springer-Verlag, Berlin. Lecture Notes in Mathematics, Vol. 470. (1975)

\bibitem{BSC1} L.A. Bunimovich, Ya. Sinai, and N. Chernov.
\emph{Markov partitions for two-dimensional billiards}.
Russ. Math. Surv. \textbf{45}, (1990), 105--152.

\bibitem{BSC2} L.A. Bunimovich, Ya. Sinai, and N. Chernov,
\emph{Statistical properties of two-dimensional hyperbolic billiards}.
Russ. Math. Surv. \textbf{46}, (1991), 47--106.

\bibitem{CNZ} M. Carney, M. Nicol and H.~K. Zhang. \emph{Compound Poisson law for hitting times to periodic orbits in two-dimensional hyperbolic systems}.   J. Stat. Phys. 169, (4), (2017), 804-823.

\bibitem{CC} J.-R. Chazottes and P. Collet. \emph{Poisson approximation for the number of visits to balls in non-uniformly hyperbolic
dynamical systems}.  Ergodic Theory and Dynamical Systems.
{\bf 33}, (2013), 49-80.

\bibitem{CM} N. Chernov and R. Markarian. \emph{Chaotic
    billiards}. Math. Surv. Monographs, \textbf{127}, AMS,
    Providence, RI, (2006), 316 pp.

\bibitem{CM07} N. Chernov and R. Markarian. \emph{Dispersing billiards
    with cusps: slow decay of correlations}. Commun. Math. Phys.,
    \textbf{270}, (2007), 727--758.

 
\bibitem{Collet} P. Collet. \emph{Statistics of closest return for some
  non-uniformly  hyperbolic  systems}. Ergodic  Theory and  Dynamical  Systems. \textbf{21},
  (2001), 401-420.

\bibitem{Coles} S. Coles. \emph{An Introduction to Statistical Modeling of Extreme Values}. Springer, 2004.

\bibitem{Dichotomy} M. Carvalho, A. C. M. Freitas, J. M. Freitas, M. Holland and M. Nicol.
\emph{Extremal dichotomy for hyperbolic toral automorphisms}.  Dynamical Systems: An International Journal, {\bf 30}, (4), (2015), 383-403.

   \bibitem{DGS} M. Denker, M. Gordin, and A. Sharova.  \emph{A Poisson limit theorem for toral automorphisms}. Illinois J. Math,  \textbf{48}, (1), (2004), 1-20.

\bibitem{Embrechts} P. Embrechts, C. Kl\"{u}pperlberg, and T. Mikosch.
  \emph{Modelling extremal events for insurance and finance}.
  Applications of Mathematics (New York), {\bf 33}, Springer-Verlag, Berlin,
  1997.

  \bibitem{sandro_coupled}  D. Faranda, H. Ghoudi, P. Guirard and  S. Vaienti. \emph{Extreme value  theory for synchronization of coupled map lattices}, Nonlinearity,  {\bf 31}, (7), (2018), 3326-3358 


\bibitem{FFLTV}  D. Faranda, J. Freitas, V. Lucarini, G. Turchetti and  S. Vaienti. \emph{Extreme value statistics for dynamical systems with noise}. Nonlinearity,  {\bf 26}, (2013), 2597.


\bibitem{FMT} D. Faranda, M. Mestre, and G. Turchetti. \emph{Analysis of round off errors with reversibility test as a dynamical indicator}. International Journal of Bifurcation and Chaos,  {\bf 22}, (9), (2012), 1250215.

\bibitem{Ferguson_Pollicott}  A. Ferguson and M. Pollicott. \emph{Escape Rates for Gibbs measures}.  Ergodic Theory and Dynamical Systems, {\bf 32}, (3), (2012), 961-988.

\bibitem{FnF}
A. C. M. Freitas and J. M. Freitas. 
On the link between dependence and independence in extreme value theory for 
dynamical systems, {\it Stat. Probab. Lett.}, {\bf 78}, (2008), 1088-1093.


\bibitem{FFRS} A. Freitas, J. Freitas, F. Rodrigues, and J. Soares. Rare events for Cantor target sets, Preprint 2019, arXiv:1903.07200.

\bibitem{FFT1}  J. Freitas, A. Freitas and M. Todd. \emph{Hitting Times and Extreme Value Theory}.
Probab. Theory Related Fields, {\bf 147}, (3), (2010). 675--710.

\bibitem{FFT3}
A. Freitas, J. Freitas, and M. Todd. Extremal Index. \emph{Hitting Time Statistics and periodicity}.
 Adv. Math., {\bf 231}, (5), (2012), 2626-2665.

\bibitem{FFT2}  A. Freitas, F. Freitas and M. Todd. \emph{Extreme value laws in
dynamical systems for non-smooth observations}.
J. Stat. Phys.  {\bf 142}, (1), (2011), 108--126.

\bibitem{FFT5}  A. C. M. Freitas, J. M. Freitas, and M. Todd. Speed of convergence for laws of rare events and escape rates. 
Stochastic Process. Appl., {\bf 125}, (4), (2015), 1653-1687.

\bibitem{FFT4} A. C. M. Freitas, J. M. Freitas, and M. Todd. \emph{The compound Poisson limit ruling periodic extreme behaviour of non-uniformly hyperbolic dynamics}. Comm. Math. Phys., \textbf{321}, (2), (2013), 483-527.

\bibitem{FHN} J. Freitas, N. Haydn and M. Nicol. \emph{Convergence of rare events point processes to the Poisson for billiards}.  Nonlinearity, \textbf{27}, (2014), 1669-1687.


\bibitem{Galambos} J. Galambos. \emph{The Asymptotic Theory of Extreme
  Order Statistics}. John Wiley and Sons, 1978.


\bibitem{Gupta} C. Gupta. \emph{Extreme-value distributions for some classes of non-uniformly partially hyperbolic
dynamical systems}.  Ergodic Theory and Dynamical Systems. {\bf 30}, (3), (2011), 757-771.

\bibitem{GHN} C.~Gupta, M.~Holland and M.~Nicol.
\emph{Extreme value theory  and return time statistics for  dispersing billiard maps and flows,
Lozi maps and Lorenz-like maps}.
Ergodic Theory and  Dynamical Systems.  {\bf 31}, (2011),  1363--1390.

\bibitem{GNO} C.~Gupta, M.~Nicol and W.~Ott. \emph{A Borel-Cantelli lemma
  for non-uniformly expanding dynamical systems}. Nonlinearity
  \textbf{23}, (8), (2010), 1991--2008.

\bibitem{HNPV} N. Haydn, M. Nicol, T. Persson and S. Vaienti. \emph{A
  note on Borel-Cantelli lemmas for non-uniformly hyperbolic
  dynamical systems}. Ergodic Theory and Dynamical Systems {\bf 33}, (2), (2013),
  475--498.
	
	\bibitem{Haydn_Vaienti} N. Haydn and S. Vaienti. \emph{Limiting entry times distribution for arbitrary null sets}, arXiv:1904.08733v1, April 18 2019.


\bibitem{Hirata} M. Hirata. \emph{Poisson Limit Law for Axiom A diffeomorphisms}. Ergodic  Theory and  Dynamical Systems. 
\textbf{13}, (3), (1993), 533--556.

\bibitem{HNT0} M. Holland, M. Nicol, A. T\"{o}r\"{o}k. \emph{Extreme value theory for non-uniformly expanding dynamical systems.}
Trans. Am. Math. Soc., {\bf 364}, (2012), 661-688.

\bibitem{HNT} M. Holland, M. Nicol, A. T\"{o}r\"{o}k. \emph{Almost
  sure convergence of maxima for chaotic dynamical systems}.
  Stochastic Process.  Appl. 126, (10), (2016), 3145--3170.

\bibitem{HRS} M. P. Holland, P. Rabassa, A. E. Sterk. Quantitative recurrence statistics and convergence to an extreme value distribution for non-uniformly hyperbolic dynamical systems. Nonlinearity, {\bf 29}, (8), (2016). 

\bibitem{HVRSB} M. Holland, R. Vitolo, P. Rabassa, A. E. Sterk, H. Broer. \emph{Extreme value laws in dynamical systems under physical observables.} Phys. D, {\bf 241}, (2012), 497-513.

\bibitem{Keller} G. Keller. \emph{Rare events, exponential hitting times and extremal indices via spectral perturbation}.
Dynamical Systems: An International Journal. \textbf{27}, (1), (2012), 11--27.

\bibitem{Kim} D.~Kim. \emph{The dynamical Borel-Cantelli lemma for
  interval maps}. Discrete Contin. Dyn. Syst. {\bf 17}, (4), (2007),
  891--900.

\bibitem{LLR} Leadbetter, M. R., Lindgren, G., and Rootzen, H.  Extremes and Related Properties of
Random Sequences and Processes. Springer-Verlag, New York, 1983.

\bibitem{V.et.al} V. Lucarini, D. Faranda, A.C. Freitas, J.M. Freitas,
  M. P.  Holland, T. Kuna, M. Nicol, M. Todd, S. Vaienti,
  \emph{Extremes and Recurrence in Dynamical Systems}, Pure and
  Applied Mathematics: A Wiley Series of Texts, Monographs, and
  Tracts, 2016.
	
	\bibitem{LFWK} V. Lucarini, D. Faranda, J. Wouters J, and T. Kuna. \emph{Towards a General Theory of Extremes for Observables of Chaotic Dynamical Systems.} Journal of Statistical Physics, {\bf  154}, (3), (2014), 723--750.
	
\bibitem{N92} H. Niederreiter. Random Number Generation and Quasi-Monte Carlo Methods, Soc. Industrial and Applied Mathematics, 1992.
	
	\bibitem{Pene_Saussol} F. P\'ene and B.  Saussol. \emph{Back to balls in billiards}. Comm. Math. Phys. {\bf 293}, (3), (2010), 837-866. 
  
\bibitem{SHRBV}	A. E. Sterk, M. P. Holland, P. Rabassa, H. W. Broer, R. Vitolo. \emph{Predictability of extreme values in geophysical models.} Nonlinear Processes in Geophysics, {\bf 19}, (2012), 529--539.
 
  \bibitem{suveges} M. S\"{u}veges, \emph{Likelihood estimation of the extremal index.} Extremes, Springer, {\bf 10}, (1-2), (2007), 
	41-55.

\bibitem{Fan_Yang} Fan Yang. \emph{Rare event processes and entry times distribution for arbitrary null sets on compact manifolds}, arxiv:1905.09956.

\bibitem{Y98} L.-S. Young.
\emph{Statistical properties of dynamical systems with some hyperbolicity}.
Ann. Math. \textbf{147}, (1998), 585--650.







\end{thebibliography}
\end{document}